\documentclass[preprint,12pt]{elsarticle}
\usepackage{amsfonts}
\usepackage{soul}
\usepackage{xcolor}
\usepackage{amssymb}
\usepackage{comment}
\usepackage{enumitem}
\usepackage{upgreek}
\usepackage{amsmath}
\usepackage{graphicx}
\usepackage{tikz-cd}
\usepackage{amscd}
\usepackage{endnotes}
\usepackage{tikz}
\usepackage{amsthm,mathtools}
\usepackage[colorlinks=true, linkcolor=blue, citecolor=blue, urlcolor=blue]{hyperref}
\usepackage[color,matrix,arrow,pdf]{xy}
\usepackage{xy}
\xyoption{all}
\usepackage{cancel}
\usetikzlibrary{positioning}
\usepackage[normalem]{ulem}

\newtheorem{theo}{Theorem}[section]
\newtheorem{cor}[theo]{Corollary}

\newtheorem{prop}[theo]{Proposition}

\theoremstyle{definition}
\newtheorem{defn}[theo]{Definition}

\theoremstyle{remark}
\newtheorem{rmk}[theo]{Remark}
\newtheorem{ex}[theo]{Example}

\numberwithin{equation}{section}

\definecolor{darkgreen}{RGB}{0,120,0}

\newcommand{\Q}{\mathcal{Q}_2}
\newcommand{\D}{\mathcal{D}_2}
\newcommand{\F}{\mathcal{F}_2}
\newcommand{\C}{\mathcal{O}_2}
\newcommand{\BB}{\mathcal{B}_2}

\newcommand{\K}{\mathcal{K}} 
\newcommand{\z}{\mathbf{Z}_2}

\newcommand{\E}{\mathbb{E}}

\newcommand{\bZ}{{\mathbb{Z}}}

\newcommand{\U}{{\mathcal{U}}}
\newcommand{\T}{{\mathcal{T}}}
\renewcommand{\O}{{\mathcal{O}}}
\newcommand{\id}{{\operatorname{id}}}

\journal{}
\makeatletter
\def\ps@pprintTitle{%
  \let\@oddhead\@empty
  \let\@evenhead\@empty
  \let\@oddfoot\@empty
  \let\@evenfoot\@empty
}
\makeatother

\begin{document}

\begin{frontmatter}



\title{Endomorphisms of the 2-adic ring $C^*$-algebra
and\\ its Weyl group}
%{Endomorphisms of the 2-adic ring C*-algebra $\mathcal{Q}_2$}
\author[a]{Dolapo Oyetunbi}
 \affiliation[a]{organization={Department of Mathematics and Statistics},
            addressline={ University of Windsor},
             postcode={ ON N9B 3P4},
%             state={Ontario},
            country={Canada}}

\author[a]{Dilian Yang} %% Author name

%\affiliation[b]{organization={Department of Mathematics and Statistics},
 %           addressline={ University of Windsor},
%             postcode={ ON N9B 3P4},
%             state={Ontario},
%country={Canada}}

%% Abstract
\begin{abstract}
In this paper, we establish a one-to-one correspondence between a natural monoid and the endomorphisms of the $2$-adic ring $C^*$-algebra $\Q$. As a consequence, we classify endomorphisms of $\Q$ with prescribed images and derive several criteria for the uniqueness of extensions of endomorphisms of a canonical copy of the Cuntz algebra $\C$ inside $\Q$ to endomorphisms of $\Q$. Moreover, we construct an example of an extendable automorphism of $\C$ that is not a composition of the flip-flop automorphism, the gauge automorphisms, and inner automorphisms, thereby providing negative answers to certain open questions. Using this explicit construction, we also show that the canonical image of $\operatorname{Aut}(\Q,\C)$ in the outer automorphism group $\operatorname{Out}(\Q)$ is non-abelian.

Finally, we characterize the automorphisms of $\Q$ that globally preserve $C^*(u)$, and completely describe the Weyl group $\mathcal{W}(\Q, C^*(u))$. Consequently, several related open questions are answered affirmatively.
\end{abstract}

\begin{keyword}
2-adic ring \(C^*\)-algebra \sep endomorphism \sep extension\sep diagonal automorphism \sep Weyl group \sep Cartan subalgebra



\MSC[2010] 46L05 

\end{keyword}

\end{frontmatter}

\section{Introduction}
Ever since Cuntz \cite{Cuntz77} introduced the Cuntz algebras $\mathcal{O}_n$ ($n\ge 2$), they have remained among the most important examples of infinite simple $C^*$-algebras and a rich source of intuition and counter-examples. Notably, Rosenberg \cite{Rosenberg77} proved that $\mathcal{O}_n$ is an amenable $C^*$-algebra that is not strongly amenable, thereby resolving a problem posed by Johnson \cite[10.2]{Johnson72}.

The Cuntz--Takesaki correspondence between the unitaries of the Cuntz algebra \(\mathcal{O}_n\) and its endomorphisms generated substantial interest in the structure of its endomorphisms (\cite{CHS12a, CHS12, CRS10, CS11, Cun80}). Applications of endomorphisms of $\mathcal{O}_n$ can be found in quantum field theory \cite{DR90} and the theory of wavelets \cite{BJ99}. Another important development in the study of automorphisms of \(\mathcal{O}_n\) is the Weyl group associated with the Cartan inclusion \(\mathcal{D}_n \subseteq \mathcal{O}_n\) (see \cite{CHS12a, CHS12} for example). Cuntz \cite{Cun80} showed that this Weyl group is discrete, and Conti and Szyma\'nski \cite{CS11} proved that it is related to the Higman--Thompson group inside \(\mathcal{O}_n\). The restricted Weyl group is studied in \cite{CHS12a}.

One way to view the $2$-adic ring $C^*$-algebra $\Q$ is as a symmetrized version of the Cuntz algebra $\C$ \cite{LL12}. Its canonical isometries satisfy the Cuntz relations together with an intertwining relation that prevents many structural phenomena possible in $\C$ from occurring in $\Q$. The canonical isometries generate a canonical copy of $\C$ inside $\Q$, and there is no conditional expectation from $\Q$ onto this copy \cite{ACR18}. This suggests a strong rigidity phenomenon for the inclusion $\C \subseteq \Q$.

Recently, Bassi and Conti \cite{BC2025} showed that the inclusion $\C \subseteq \Q$ is rigid, while in \cite{OY26} we proved that there are no proper intermediate $C^*$-subalgebras between $\C$ and $\Q$. Earlier, Aiello, Conti, and Rossi studied permutative endomorphisms of $\C$ that extend to endomorphisms of $\Q$ \cite{ACR20} and characterized certain classes of endomorphisms that fix specified generators of $\Q$ \cite{ACR18}. 

This paper addresses several important questions concerning endomorphisms and a Weyl group of $\Q$ that were posed in \cite{ACR21}.

Motivated by the Cuntz--Takesaki correspondence, we establish a one-to-one correspondence between a certain monoid and the endomorphisms of the $2$-adic ring $C^*$-algebra $\Q$. As a consequence, we obtain two classifications for endomorphisms of $\Q$. One of them generalizes \cite[Theorems 6.14 \& 6.18]{ACR18}.

\begin{theo}[See Theorem \ref{P:u-uni}]\label{P:u-uni_intro}
Let $\lambda_{(U, V)}$ be the endomorphism of $\mathcal{Q}_2$ determined by $(U,V)$. Then any endomorphism $\lambda\in \operatorname{End}(\mathcal{Q}_2)$ satisfies
$\lambda(u)=U$ if and only if $\lambda=\lambda_{(U,  XV)}$ for some unitary $X\in C^*(U)'\cap \mathcal{Q}_2$. 
\end{theo}

Next, we provide several criteria for the uniqueness of extensions of endomorphisms of $\C$ to $\Q$. In particular, we show that every permutative endomorphism of $\C$ as well as every endomorphism that globally preserves the diagonal $\D$, admits at most one extension to an endomorphism of $\Q$.

\begin{theo}[See Theorem \ref{thm:uniqueness_criteria2} and Proposition \ref{prop:uniqueness}]
Let $\lambda_w \in \operatorname{End}(\C)$ satisfy one of the following conditions:
\begin{enumerate}
\item $\lambda_w(\F) \subseteq \F$,
\item $\lambda_w(\D) = \D$.
\end{enumerate}
Then $\lambda_w$ admits at most one extension to an endomorphism of $\Q$. Moreover, if $\lambda_w$ has a unique extension $\tilde{\lambda}_w$, then
\[
\lambda_w(\mathcal{O}_2)' \cap \mathcal{Q}_2 \subseteq C^*(\tilde{\lambda}_w(u))' \cap \mathcal{Q}_2 .
\]
\end{theo}

In \cite{ACR18b}, Aiello, Conti, and Rossi proved that every extendable localized diagonal automorphism of $\C$ is exactly the product of a gauge automorphism and a localized inner diagonal automorphism. They further observed that their methods could not be extended to the general setting. We resolve this question by showing that the above characterization fails for general extendable diagonal automorphisms. This construction provides negative answers to several related open questions in \cite{ACR21}.

\begin{theo}[See Theorem \ref{T:lambdad}]\label{T:lambdad1}
There exists a unitary \( d \in \D \) such that the corresponding automorphism \( \lambda_d \in \operatorname{Aut}(\C) \) satisfies the following properties:
\begin{enumerate}
\item \( \lambda_d \) extends to an automorphism of \( \Q \);
\item \( \lambda_d(x) = x \) for all \( x \in C^*(\D, s_2) \);
\item \( \lambda_d \) does not lie in the subgroup of \( \operatorname{Aut}(\C) \) generated by the flip-flop \( \sigma \), 
the gauge automorphisms \( \alpha_t \), and inner automorphisms.
\end{enumerate}
\end{theo}

Using the extension of $\lambda_d$ constructed in Theorem~\ref{T:lambdad1}, we obtain the following property of an important subgroup of the outer automorphism group of $\Q$. 
\begin{theo}[See Theorem \ref{T:AutQOnonabe}]
The canonical image of $\operatorname{Aut}(\Q, \C)$ in $\operatorname{Out}(\Q)$ is non-abelian.
\end{theo}

Another important, but previously overlooked, Cartan subalgebra of $\Q$ is $C^*(u)$ (see Proposition~\ref{prop:cartan2}). Building on Theorem~\ref{P:u-uni_intro} and results from \cite{Komura25}, we completely characterize the automorphisms of $\Q$ that globally preserve $C^*(u)$ and obtain a description of the associated Weyl group $\mathcal{W}(\Q, C^*(u))$. As applications, we provide a complete description of the unitary normalizer of $C^*(u)$ and the topological full group of the Deaconu--Renault groupoid $\Gamma(\mathbb T, \gamma)$ associated with the Cartan pair $(\Q, C^*(u))$. In particular, we prove the following theorem.

\begin{theo}[See Theorem \ref{T:autp}, Corollaries \ref{C:normalizer}, \ref{C:Weyl2} and \ref{C:TFG}]
\label{T:Weyl}
Let $\mathbb{Z}(2^\infty)$ denote the Pr\"ufer 2-group. Then
\[
\operatorname{Aut}(\Q, C^*(u))
=
\left\{
\sigma^{n}\circ \operatorname{Ad}(U_{\xi})\circ \beta_{V}
:\,
\xi\in\mathbb{Z}(2^\infty),\ V\in\mathcal{U}(C^*(u)),\ n=0,1
\right\}
\]
and
\[
\operatorname{Inn}(\Q, C^*(u))
=
\left\{
\operatorname{Ad}(U_{\xi}V)
:\,
\xi\in\mathbb{Z}(2^\infty),\ V\in\mathcal{U}(C^*(u))
\right\}.
\]
Consequently, we obtain
\begin{enumerate}[label=(\roman*), leftmargin=*]
\item the Weyl group
\[
\mathcal{W}(\Q, C^*(u))
\cong
\mathbb{Z}(2^\infty)\rtimes \mathbb{Z}/2\mathbb{Z};
\]
\item the unitary normalizer group
\[
\mathcal{N}_{C^*(u)}(\Q)
=
\left\langle
\mathcal{U}(C^*(u)),\,
U_{\xi}:\xi\in\mathbb{Z}(2^\infty)
\right\rangle;
\]
\item the topological full group
\[
[[\Gamma(\mathbb{T},\gamma)]]
\cong
\mathbb{Z}(2^\infty).
\]
\end{enumerate}
\end{theo}

\noindent \textbf{Structure of the paper.}
After recalling the necessary notation and background material in Section~\ref{S:Pre}, we investigate the structure of $\operatorname{End}(\Q)$ in Section~\ref{S:fundamental}. Section~\ref{S:unique} is devoted to necessary or sufficient conditions for the uniqueness of extensions of endomorphisms of $\C$ to $\Q$. In Section~\ref{S:diagaut}, we construct an explicit diagonal automorphism that is not a composition of the previously known concrete extendable automorphisms of $\C$. This construction allows us to identify a non-abelian subgroup of the outer automorphism group of $\Q$.  Finally, in Section~\ref{S:Weyl}, we completely characterize the automorphisms of $\Q$ that globally preserve $C^*(u)$ and determine the structure of the associated Weyl group. As a by-product, we obtain a description of the unitary normalizer group of $C^*(u)$ in $\Q$ and the topological full group of the associated Deaconu--Renault groupoid.

\bigskip
\noindent
\textbf{Acknowledgements.} 
This work is partially supported by an NSERC Discovery Grant.

\section{Preliminaries}
\label{S:Pre}
This section briefly recalls some necessary background on the 2-adic ring C*-algebra $\Q$ and 2-adic integers $\z$. Refer to \cite{ACR21, Fol95, LL12, OY26, Rob00, Tailbleson67} for more information.

\begin{defn}\label{defn:2-adic}
The \emph{2-adic ring $C^*$-algebra} $\mathcal{Q}_2$ is the universal unital $C^*$-algebra generated by a unitary $u$ and an isometry $s_2$ satisfying the following relations:
$$
\text{(I)} \quad s_2 u = u^2 s_2, \qquad
\text{(II)} \quad s_2 s_2^* + u s_2 s_2^* u^{*} = 1.
$$
\end{defn}
%%%%%%%%%%%

By the universal property of $\Q$, for each $t\in \mathbb{T}$ there is $\alpha_t\in \mathrm{Aut}(\Q)$ defined by 
$
\alpha_t(u) = u\text{ and }\alpha_t(s_2) = t s_2. 
$
This yields an action of $\mathbb{T}$ on $\Q$, which is known as the \textit{gauge action} of $\Q$. 

It is shown that $\Q$ is a Kirchberg algebra satisfying the UCT with \(K_0 (\Q) = K_1 (\Q) = \mathbb{Z}\) (\cite{Katsura08, LL12}).

%%%%%%%%%%%
The $C^*$-algebra $\Q$ has several realizations in the literature, some of which are implicit or presented in different forms. For examples, see \cite{BRRW14, CaHR16, Katsura08, LL12, LY19, Spi12}. For our purposes, here we recall only three of these realizations.  

\begin{enumerate}
\item
$\Q$ is the C*-algebra of a self-similar graph $(\bZ, \textsf{E}_2)$ (see \cite{LY21} for example): $\Q\cong \O_{\bZ, \textsf{E}_2}$, where

\[
\small
\mathsf{E}_2= 
\xymatrix{
\bullet\ar@[black]@(ur,dr)^{\textsf{e}_{2}}\ar@[black]@(dl,ul)^{\textsf{e}_1}
}
\]
and 
\[
1\cdot \textsf{e}_1=\textsf{e}_2,\ 1\cdot 
\textsf{e}_2=\textsf{e}_1,\ 1|_{\textsf{e}_1}=1, \ 1|_{\textsf{e}_2}=0. 
\]
This realization facilitates some proofs. 

\item $\Q$ is a semigroup crossed product (\cite{LL12}): $\Q\cong C(\mathbf{Z}_2)\rtimes (\bZ\rtimes [2\rangle)$,
where $\mathbf{Z}_2$ is the 2-adic integers. To the best of the authors' knowledge, this explains why $\Q$ is known as the ``2-adic ring $C^*$-algebra.''

\item $\Q$ has a canonical representation (\cite{LL12}): On 
$\ell^2(\bZ)=\overline{\operatorname{span}}\{e_n: n\in\mathbb{Z}\}$,
\[
u(e_n):=e_{n+1},\ s_1(e_n):=e_{2n+1},\ s_2(e_n):=e_{2n} \quad \text{for all }n\in \bZ.
\]
The canonical representation was first introduced in \cite{LL12} and has since been used extensively in the study of $\Q$. See the survey \cite{ACR21} and the references therein for further details.
\end{enumerate}

%%%%%%%%%%%

The $C^*$-algebra $\Q$ contains several notable $C^*$-subalgebras that will be used frequently throughout the paper. These include the canonical copy $\C = C^*(s_1,s_2)$ of the Cuntz algebra, where $s_1 := us_2$; the diagonal subalgebra $\D$, which is also a Cartan subalgebra of $\Q$; the fixed-point algebra $\mathcal{B}_2$ of the gauge action $\alpha$; the fixed-point algebra $\F$ of $\alpha|_{\C}$; and $C^*(u)$, which is a MASA in $\Q$ \cite{ACR18}.

It is known that $\mathcal{B}_2$ and $\mathcal{F}_2$ are a Bunce-Deddens algebra and UHF algebra of type $2^\infty$, respectively  (\cite{BOS18, VY25}). So both of them have a unique faithful tracial state. The spectrum of \(\mathcal{D}_2\) is homeomorphic to \(\z\) (\cite{LL12}).
Very recently, in \cite{OY26} we show that $\C$ is a maximal subalgebra of $\Q$ -- there are no non-trivial intermediate $C^*$-subalgebras between $\C$ and $\Q$. In other words, the inclusion $\C \subseteq \Q$ is tight in the terminology of \cite{BC2025}. As a consequence, we resolve several open questions posed in \cite{ACR21}.

By an endomorphism of a unital $C^*$-algebra $A$, we always mean a unital $^*$-endomorphism of $A$. We denote the set of all such endomorphisms by $\operatorname{End}(A)$. Let $B\subseteq A$ be a unital $C^*$-subalgebra. We denote the automorphism group of $A$ by $\operatorname{Aut}(A)$, the subgroup of automorphisms that globally preserve $B$ by
\[
\operatorname{Aut}(A, B)=\{\phi\in\operatorname{Aut}(A):\phi(B)=B\},
\]
and the subgroup of automorphisms that fix $B$ pointwise by
\[
\operatorname{Aut}_B(A)=\{\phi\in\operatorname{Aut}(A):\phi(x)=x \text{ for all } x\in B\}.\]

Two distinguished endomorphisms of $\Q$ are of particular importance, as they will play a central role later. The first is the flip-flop automorphism $\sigma$ defined by
\[
\sigma(s_1)=s_2,\qquad \sigma(s_2)=s_1,\qquad \sigma(u)=u^*,
\]
which is the unique extension of the well-known flip-flop automorphism of $\C$. In \cite{OY26} we provide two concrete descriptions of the fixed-point algebra of $\sigma$. Analogous to the phenomenon for $\C$ studied in \cite{CL12}, both the fixed-point algebra and the crossed product induced by $\sigma$ turn out to be isomorphic to $\mathcal{Q}_2$. Moreover, the two corresponding inclusions are tight as well. See \cite{OY26} for further details.
 
The second one is a proper endomorphism on $\Q$, which is the unique extension of the canonical shift $\varphi$ on $\C$, still denoted as $\varphi$: 
\[
\varphi(x)=s_1 x s_1^*+s_2x s_2^*\quad \text{for all }x\in \Q.
\]
In particular, $\varphi(u)=u^2$. This endomorphism and its variants play a vital role in \cite{OY26}.

Let us end this section by reviewing some basics on 
the 2-adic integers $\z$. 
Every element $x$ in the 2-adic field $\mathbb{Q}_2$ has a unique (aka canonical) expansion 
\[
x=\sum_{n=N}^\infty a_n 2^n,
\]
where $N\in \bZ$, each $a_n\in\{0,1\}$,  and $a_N\ne 0$. 
The fractional part of $x$, denoted by $\{x\}$, is 
\[
\{x\}:=\sum_{n=N}^{-1} a_n 2^n. 
\]
So $\z$ is the subgroup of the `integers' of $\mathbb{Q}_2$. Thus every element $k\in \z$ has a unique (canonical) expansion 
\[
k=\sum_{n=0}^\infty a_n 2^n,
\]
where $a_n\in\{0,1\}$. 

Let $\mathbb{Z}(2^\infty)$ be the Pr\"ufer 2-group: 
\[
\mathbb{Z}(2^\infty)=\{z\in \mathbb{T}\mid z^{2^n}=1\text{ for some }0\le n\in \mathbb{Z}\},
\]
the group of all $2^n$-th power roots of unity in $\mathbb{C}$. Note that the Pontryagin dual group $\widehat{\mathbb{Z}}_2$ of $\z$ is isomorphic to $\mathbb{Z}(2^\infty)$ \cite{Som20}.

The Pontryagin dual group $\widehat{\mathbb{Q}}_2$ is isomorphic to itself \cite{Fol95}.
In fact, let $\chi$ be the non-trivial character of $\mathbb{Q}_2$ defined by
\[
\chi(x)=e^{2\pi i\{x\}}\quad\text{for all }x\in \mathbb{Q}_2.
\]
 Then every character of $\mathbb{Q}_2$ is of the form $\chi_u(x)=\chi(ux)$ for some $u\in \mathbb{Q}_2$. 

\section{Some fundamentals on $\operatorname{End}(\Q)$}
\label{S:fundamental}

In this section, we study some fundamental properties of the endomorphisms of $\Q$. Probably the most important one is the one-to-one correspondence between the set of all endomorphisms of $\Q$ and a certain monoid. This correspondence plays a key role throughout the paper. As one of its applications, we classify two classes of endomorphisms of $\Q$, which generalize \cite[Theorems 6.14 and 6.18]{ACR18}.
Further applications of Theorem \ref{T:endZ} will be given in Section \ref{S:unique}.
 
\subsection{A basic structure of $\operatorname{End}(\Q)$} 

First, we recall the well-known Cuntz-Takesaki correspondence on endomorphisms of $\C$ (see \cite{Cun80}):
There is a one-to-one correspondence 
%\begin{align*}
$
\U(\O_2)\to  \operatorname{End}(\O_2),\ V\mapsto \lambda_V, 
%\end{align*}
$
where 
\[
\lambda_V(s_i)= Vs_i\ (i=1,2).
\]
Its inverse is given by 
\[
\lambda\in \operatorname{End}(\O_2) \mapsto 
V:=\lambda(s_1)s_1^*+\lambda(s_2)s_2^*\in \U(\O_2).
\]

Motivated by the above and by the structure of endomorphisms of rank-2 graph $C^*$-algebras \cite{Yan10} (cf.~\cite{VY25}), we now focus on the endomorphisms of $\Q$. Before stating our result, let us set up some notation. For a unital $C^*$-algebra $A$, we denote the set of its unitary elements by $\mathcal{U}(A)$.
Given a unitary $U\in \U(\Q)$, let us define
\[
W_U:=s_1s_2^*+s_2Us_1^*.
\]
Then $W_U$ is a unitary in $\Q$.

For $A,B\in\Q$, we denote by $\T(A,B)\subseteq\Q$ the set of elements of $\Q$ intertwining $A$ and $B$, namely
\[
\T(A,B)=\{X\in\Q: AX=XB\}.
\]
The following result will be used frequently throughout the paper. Due to its importance, we state it as a theorem, although its proof is straightforward.

\begin{theo}
\label{T:endZ}
Let 
\[
\mathcal{S}_{\Q}:=\left\{(U,V)\in \U(\Q)\times \U(\Q): V\in \T(U, W_U)\right\}.
\]
\begin{itemize}
\item[(i)] 
There is a one-to-one correspondence 
\begin{align*}
\Xi: \mathcal{S}_{\Q}&\to \operatorname{End}(\Q),\ (U,V)\mapsto \lambda_{(U,V)},
\end{align*}
where $\lambda_{(U,V)}: u\mapsto U,\ s_i\mapsto Vs_i \ (i=1,2)$.  

\item[(ii)] 
$\mathcal{S}_{\Q}$ has a monoid structure with identity $(u, I)$, which makes $\Xi$ in (i) a monoid isomorphism. 

\item[(iii)] 
$\lambda_{(U,V)}\in \operatorname{End}(\Q)$ is an automorphism if and only if both $u$ and $V$ belong to $\lambda_{(U,V)}(\mathcal{Q}_2)$.
\end{itemize}
\end{theo}

\begin{proof}
(i) Let $(U,V)\in\mathcal{S}_{\Q}$. Then define $\lambda_{(U,V)}:\Q\to \Q$ via 
\[
\lambda_{(U,V)}(u):=U, \quad \lambda_{(U,V)}(s_i):=V s_i =: S_i\ (i=1,2). 
\]
One can verify that 
\begin{align*}
U S_1=UV s_1=VW_U s_1=V s_2U=S_2 U,\\
U S_2=UV s_2=VW_U s_2=Vs_1=S_1.
\end{align*}
Therefore, by the universal property of $\Q$, $\lambda_{(U,V)}$ is an endomorphism of $\Q$. 

Conversely, let $\lambda$ be an endomorphism on $\Q$. Let 
\[
U=\lambda(u), \  
\ V= \lambda(s_1)s_1^* + \lambda(s_2)s_2^*.
\]
Then one can easily check that
$
 V\in \T(U, W_U),
 $ 
 and that $\lambda$ agrees with $\lambda_{(U,V)}$ on the generators $u,s_1,s_2$. Hence $\lambda=\lambda_{(U,V)}$.
 
Clearly, the assignments
\[
(U,V)\mapsto \lambda_{(U,V)}
\]
and
\[
\lambda\mapsto
\bigl(\lambda(u),\lambda(s_1)s_1^*+\lambda(s_2)s_2^*\bigr)
\]
are inverse to each other.

%%%%%%%%%%
(ii) For $(U_i, V_i)\in \mathcal{S}_{\Q}$ $(i=1,2)$, let $\lambda_i:=\lambda_{(U_i,V_i)}\in \operatorname{End}(\Q)$. Define 
\[
(U_2, V_2)*(U_1, V_1):=(\lambda_2(U_1), \lambda_2(V_1)V_2).
\]
Then simple calculations show that $(\mathcal{S}_{\Q}, *)$ is a monoid with identity $(u,I)$.

(iii) By the proof of (ii) above, one has 
\begin{align*}
(U_2, V_2)*(U_1, V_1)=(u, I)
&\iff (\lambda_2(U_1), \lambda_2(V_1)V_2)=(u, I)\\
&\iff
\lambda_2(U_1)=u \text{ and } \lambda_2(V_1)=V_2^*. 
\end{align*}
Hence $\lambda_2$ is an automorphism if and only if both $u$ and $V_2^*$ (equivalently $V_2$) belong to $\lambda_2(\mathcal{Q}_2)$.
\end{proof}

Some remarks are in order. 
%%%%%%%%%%

\begin{rmk}
Writing out $V\in \T(U, W_U)$ yields the following identity satisfied by $U$ and $V$: 
\begin{align}
\label{E:endo_criteria}
U=S_1S_2^*+S_2US_1^*
\end{align}
where $S_i:=Vs_i$ $(i=1,2$).     
\end{rmk}
%%%%%%%%
\begin{rmk}
It is worth mentioning that Theorem \ref{T:endZ} can be 
generalized to a large class of self-similar graph C*-algebras.  
\end{rmk}
%%%%%%%%%%
\begin{ex}
The flip-flop automorphism $\sigma$ and the canonical shift endomorphism $\varphi$ can be identified with
\begin{equation}
\sigma \coloneqq\lambda_{(u^*,\, F)}\qquad \text{and}\qquad \varphi \coloneqq \lambda_{(u^2,\, V)}
\end{equation}
where $F=s_1s_2^*+s_2s_1^*$ and  
$V=\sum_{i,j=1}^2 s_{ij}s_{ji}^*$.
\end{ex}
%%%%%%%%%%%%%%%%%

\subsection{Two classifications}
In this subsection, we apply Theorem \ref{T:endZ} to classify two classes of endomorphisms. The first class consists of endomorphisms that share the same image of the generator $u$. As an immediate consequence, we significantly simplify some results in \cite{ACR18}. The second classification provides an obstruction to the uniqueness of extensions. These classification results will be used in later sections. 
%Further applications of Theorem \ref{T:endZ} will be given in Subsection %\ref{SS:app}.

Below is the first classification result.

\begin{theo}
\label{P:u-uni}
Let $\lambda_{(U, V)}$ be the endomorphism of $\mathcal{Q}_2$ determined by $(U,V)$. Then any endomorphism $\lambda\in \operatorname{End}(\mathcal{Q}_2)$ satisfies
$\lambda(u)=U$ if and only if $\lambda=\lambda_{(U,  XV)}$ for some unitary $X\in C^*(U)'\cap \mathcal{Q}_2$. 
\end{theo}

\begin{proof}
Let $(U, \widetilde V)\in \mathcal{S}_{\Q}$ be the pair determined by $\lambda$. Then 
\begin{align*}
V^*UV=s_1s_2^*+s_2U s_1^*,\quad \widetilde V^*U\widetilde V=s_1s_2^*+s_2U s_1^*.
\end{align*}
So $V^*UV=\widetilde V^*U\widetilde V$ implies $\widetilde V V^* U=U \widetilde V V^*$. Thus $\widetilde V= X V$ for some $X\in C^*(U)'\cap \mathcal{Q}_2$. 
The converse follows immediately from Theorem \ref{T:endZ}.
\end{proof}

As a quick remark, the extension of an extendable endomorphism $C^*(u)\to C^*(u)$, $u\mapsto U$, to $\mathcal{Q}_2$ is never unique since $C^*(U)\subseteq C^*(u)$ is abelian.
%%%%%%%%%

\medskip
Next, we recover \cite[Theorems 6.14 and 6.18]{ACR18} as consequences of Theorem \ref{P:u-uni} and obtain even more.
\begin{cor}
\label{C:u-uni}
Let $V\in \U(\Q)$. 
\begin{itemize}
\item[(i)] $\lambda_{(u, V)}\in \operatorname{End}_{C^*(u)}(\mathcal{Q}_2)\iff\lambda_{(u, V)}\in \operatorname{Aut}_{C^*(u)}(\mathcal{Q}_2)\iff V\in \U(C^*(u))$. In particular, $\operatorname{Aut}_{C^*(u)}(\mathcal{Q}_2) \cong \mathcal{U}(C^*(u))$ as groups. 

\item[(ii)] $\lambda_{(u^*, V)}\in \operatorname{End}(\mathcal{Q}_2, C^*(u))\iff\lambda_{(u^*, V)}\in \operatorname{Aut}(\mathcal{Q}_2, C^*(u))\iff$ $V=XF$ for some unitary $X$ in $C^*(u)$, where $F$ is the flip-flop unitary. 
\end{itemize} 
\end{cor}

\begin{proof}
(i) Applying Theorem \ref{P:u-uni} to $\lambda_{(u,I)}=\id$, we get that
$\lambda_{(u,V)}\in \operatorname{End}_{C^{*}(u)}(\mathcal{Q}_2)$ if and only if $V\in C^*(u)$. To show that $\lambda_{(u,V)}$ is an automorphism for every $V\in \mathcal{U}(C^*(u))$, it suffices, by Theorem \ref{T:endZ}, to show that
$u,V\in \lambda_{(u,V)}(\mathcal{Q}_2)$. Since
$u\in \lambda_{(u,V)}(\mathcal{Q}_2)$, it follows that
$V\in \lambda_{(u,V)}(\mathcal{Q}_2)$, and the conclusion follows.

(ii) Applying Theorem \ref{P:u-uni} to $\lambda_{(u^*,F)}(=\sigma)$, we obtain that
$\lambda_{(u^*,V)}\in \operatorname{End}(\mathcal{Q}_2,C^*(u))$ if and only if
$V=XF$ for some unitary $X\in C^*(u)$. By Theorem \ref{T:endZ}, one can easily verify
\[
\sigma\circ\lambda_{(u,\sigma(X))}=\lambda_{(u^*,XF)}.
\]
Since $\sigma$ is an automorphism and
$\lambda_{(u,\sigma(X))}\in \operatorname{Aut}(\Q,C^*(u))$ by (i) above,
it follows that $\lambda_{(u^*,XF)}$ is also an automorphism.
\end{proof}
\begin{rmk}
\label{R:u-uni}
In \cite[Theorem 6.14]{ACR18}, the authors showed that there is a (group) isomorphism 
\[
C(\mathbb{T}, \mathbb{T}) \to \operatorname{Aut}_{C^*(U)}(\Q)
\]
given by $f \mapsto \beta_f$, where
\[
\beta_f(u) = u, \qquad \beta_f(s_2) = f(u)\, s_2.
\]
%%%%%
But one can easily verify that $\beta_f=\lambda_{(u,f(u))}$. In fact, if $\beta_f=\lambda_{(u,V)}$, then
\begin{align*}
V
=\beta_f(us_2)s_1^*+f(u)s_2s_2^*
=f(u)s_1s_1^*+f(u)s_2s_2^*=f(u).
\end{align*}
Therefore, Corollary \ref{C:u-uni} is a restatement of Theorem 6.14 in \cite{ACR18}. Accordingly, we write $\beta_V$ in place of $\beta_f$. 
\end{rmk}

%%%%%%%%%%%%%%%%%Rewrote
Next, we classify the second class which share the same $V$.  

%%%%%%%%%%%%%%%%%
\begin{prop}
\label{P:Vclassifi} 
Let $(U, V)\in \mathcal{S}_{\Q}$. 
Then $(\widetilde U, V)\in \mathcal{S}_{\Q}$ if and only if 
$\widetilde U=UW$ (resp.~$\widetilde W U$),
where $W$ (resp.~$\widetilde W$) satisfy 
\begin{align}
\label{E:u1*u2}
V^*WV&=s_2s_2^*+s_1W s_1^*,\\
\label{E:u1u2*}
(\text{resp.~}V^*\widetilde WV&=s_1s_1^*+s_2\widetilde W s_2^*). 
\end{align}
\end{prop}

\begin{proof}
By Theorem \ref{T:endZ} one has 
\begin{align}
\label{E:U}
V^*UV&=s_1s_2^*+s_2U s_1^*,\\
\label{E:tildeU}
V^*\widetilde UV&=s_1s_2^*+s_2\widetilde U s_1^*.
\end{align}
Then 
\begin{align*}
V^*U^*\widetilde U V&=s_2s_2^*+s_1U^*\widetilde U s_1^*,\\ 
V^*U\widetilde U^*V&=s_1s_1^*+s_1U\widetilde U^*s_2^*. 
\end{align*}
Thus $W:=U^*\widetilde U$ satisfies Eq.~\eqref{E:u1*u2}. 

Conversely, if $\widetilde U=UW$ where satisfies Eq.~\eqref{E:u1*u2}. Then 
multiplying both sides of Eq.~\eqref{E:U} and Eq.~\eqref{E:u1*u2} shows that $(UW, V)\in \mathcal{S}_{\Q}$. 

The proof of the other case is similar. 
\end{proof}

%%%%%%%%%%%%%%%%%
A quick corollary of Proposition \ref{P:Vclassifi} is the following characterization of extensions of endomorphisms of $\C$ to $\Q$. When we require $V$ in Proposition \ref{P:Vclassifi} to be a unitary in $\C$, this characterization highlights an obstruction to the uniqueness of extensions of extendable endomorphisms of $\C$.
%%%%%%%%%%%%%%%%%
\begin{cor}
\label{C:Vclassifi}
Let $V\in\mathcal{U}(\mathcal{O}_2)$. Suppose that $\lambda_V\in \operatorname{End}(\C)$ is extendable to $\lambda_{(U,V)}\in \operatorname{End}(\mathcal{Q}_2)$. 
Then every extension of $\lambda_V$ is given by $(UW, V)$ where $W$ satisfies Eq.~\eqref{E:u1*u2}. 
Equivalently, 
every extension of $\lambda_V$ is given by $(\widetilde W U, V)$ where $\widetilde W$ satisfies Eq.~\eqref{E:u1u2*}.

In particular, $\lambda_V$ has a unique extension if and only if one (equivalently, both) of Eq.~\eqref{E:u1*u2} and Eq.~\eqref{E:u1u2*} has only the trivial solutions $W=\widetilde W=I$.
\end{cor}

%%%%%%%%%%%%%%%%%
\section{Uniqueness of extensions of endomorphisms of $\C$}
\label{S:unique}

Let $\lambda_V\in\operatorname{End}(\C)$ be an extendable endomorphism.
In this section, we first provide some sufficient or necessary conditions for the uniqueness of its extension to an endomorphism on $\Q$. 
Then we focus on the case that $\lambda_V$ is an automorphism. 
We show that $\lambda_V$ is an automorphism if and only if its extension is an automorphism of $\Q$, and furthermore, the extension has to be unique. 
This answers some open questions posed in \cite{ACR18, ACR21}. It is worth mentioning that some of these results are obtained as simple applications of maximality in \cite{OY26} with the help of the fundamental properties of the endomorphisms of $\Q$ discussed in Section \ref{S:fundamental}.

\subsection{Some sufficient/necessary conditions for unique extensions}

We begin with a fairly general result about the uniqueness of extendable endomorphisms of $\C$ to $\Q$. 

%%%%%%%%%%%%%%%%%%%%%Added on Apr 6
\begin{theo}\label{thm:uniqueness_criteria2}
Let $\lambda_w \in \operatorname{End}(\C)$ be extendable. Suppose one of the following conditions holds:
\begin{enumerate}
\item $\lambda_w(\F) \subseteq \F$;
\item there exists a unital $C^*$-subalgebra $B \subseteq \lambda_w(\C)$ containing $\lambda_w(s_1^n s_1^{*n})$ for all $n \ge 1$, together with a faithful conditional expectation $\mathcal{E}: \Q \to B$.
\end{enumerate}
Then $\lambda_w$ admits a unique extension to $\Q$.
\end{theo}

\begin{proof}
Let $\tau$ be the unique faithful trace on $\BB$, and $\E : \Q \to \BB$ 
%and $\mathcal{E} : \Q \to B$ be
the the faithful conditional expectation induced from the gauge action $\alpha$. Define $\omega : \lambda_w(\C) \to \mathbb{C}$ by $\omega(\lambda_w(x)) = \tau(x)$; this is a faithful state on $\lambda_w(\C)$. Then both $\tau_1 \coloneqq \tau \circ \E$ and $\omega_1 \coloneqq \omega \circ \mathcal{E}$ are faithful states on $\Q$.

Let $\tilde{\lambda}_1, \tilde{\lambda}_2 : \Q \to \Q$ be two extensions of $\lambda_w$, and set $T_i := \lambda_w(s_i)$. Repeated use of
\[
\tilde{\lambda}_i(u) = T_1 T_2^* + T_2 \tilde{\lambda}_i(u) T_1^*
\]
from \eqref{E:endo_criteria} yields
\begin{equation}
\label{E:4.1}
\tilde{\lambda}_i(u)
= \sum_{k=0}^{n} T_2^k T_1 T_2^* T_1^{*k}
+ T_2^{n+1} \tilde{\lambda}_i(u) T_1^{*(n+1)}, \qquad i=1,2.
\end{equation}

For a faithful state $\psi$ on $\Q$, define $\|a\|_{\psi} \coloneqq \sqrt{\psi(a^*a)}$. Then it follows from \eqref{E:4.1} that 
\begin{align*}
\|\tilde{\lambda}_1(u) - \tilde{\lambda}_2(u)\|_{\psi}
&\le \left\|\tilde{\lambda}_1(u) - \sum_{k=0}^{n} T_2^k T_1 T_2^* T_1^{*k}\right\|_{\psi}
+ \left\|\sum_{k=0}^{n} T_2^k T_1 T_2^* T_1^{*k} - \tilde{\lambda}_2(u)\right\|_{\psi} \\
&\le \|T_2^{n+1}\tilde{\lambda}_1(u)T_1^{*(n+1)}\|_{\psi}
+ \|T_2^{n+1}\tilde{\lambda}_2(u)T_1^{*(n+1)}\|_{\psi} \\
&=2\, \sqrt{\psi(T_1^{n+1}T_1^{*(n+1)})}
= 2\,\sqrt{(\psi \circ \lambda_w)(s_1^{n+1}s_1^{*(n+1)})}.
\end{align*}

%\noindent
(i) Take $\psi := \tau_1$. Since $\lambda_w(\F)\subseteq \F$, the state $\tau_1 \circ \lambda_w$ restricts to a faithful trace on $\F$. By uniqueness of the trace on $\F$,
\[
(\tau_1 \circ \lambda_w)(s_1^{n+1}s_1^{*(n+1)}) = \tau(s_1^{n+1}s_1^{*(n+1)}) = 2^{-(n+1)}.
\]
Hence $\|\tilde{\lambda}_1(u) - \tilde{\lambda}_2(u)\|_{\psi} \to 0$, and so $\tilde{\lambda}_1(u) = \tilde{\lambda}_2(u)$.

%\noindent
(ii) Take $\psi := \omega_1$. Then it follows from the conditional expectation $\mathcal{E}$ and defnition of $\omega_1$ that
\[
(\psi \circ \lambda_w)(s_1^{n+1}s_1^{*(n+1)})
= \omega(\lambda_w(s_1^{n+1}s_1^{*(n+1)}))
= \tau(s_1^{n+1}s_1^{*(n+1)})
= 2^{-(n+1)}.
\]
Thus $\|\tilde{\lambda}_1(u) - \tilde{\lambda}_2(u)\|_{\psi} \to 0$, and again $\tilde{\lambda}_1(u) = \tilde{\lambda}_2(u)$.

Therefore the extension is unique for both cases.
\end{proof}
\begin{cor}
Let $\lambda_w \in \operatorname{End}(\C)$ be extendable with $\lambda_w(\D) = \D$. Then $\lambda_w$ admits a unique extension to $\Q$.
\end{cor}

\begin{proof}
Take $B = \D$ and $\mathcal{E}: \Q \to \D$ the canonical faithful conditional expectation in Theorem \ref{thm:uniqueness_criteria2}.
\end{proof}

As a by-product, Theorem \ref{thm:uniqueness_criteria2} (a uniqueness result) yields a necessary condition for extendability. 

\begin{cor}
Let $w \in \F$ and $\lambda_w \in \operatorname{End}(\C)$ be extendable to $\tilde{\lambda} \in \operatorname{End}(\Q)$. Then $\tilde{\lambda}(u) \in \BB$.
\end{cor}

\begin{proof}
Recall that $\mathcal{B}_2$ is the fixed point algebra of the gauge acton $\alpha$ of $\Q$ on $\mathbb{T}$.
Suppose, for contradiction, that $\tilde{\lambda}(u) \notin \BB$. Then there exists $t \in \mathbb{T}$ such that 
\(
\alpha_t(\tilde{\lambda}(u)) \neq \tilde{\lambda}(u).
\)
%where $\alpha_t$ denotes the standard gauge automorphism of $\Q$.  
Using \eqref{E:endo_criteria}, we have
\[
\tilde{\lambda}(u) = T_1 T_2^* + T_2 \tilde{\lambda}(u) T_1^*,
\]
where $T_i = \lambda_w(s_i) = w s_i$ ($i=1,2$).
Applying $\alpha_t$, we get
\[
\alpha_t(\tilde{\lambda}(u)) = T_1 T_2^* + T_2 \alpha_t(\tilde{\lambda}(u)) T_1^*.
\]
Then, by Theorem \ref{T:endZ}, the map $\tilde{\lambda}_1 : \Q \to \Q$ defined by
\[
\tilde{\lambda}_1(u) := \alpha_t(\tilde{\lambda}(u)), \qquad \tilde{\lambda}_1(s_i) := w s_i
\]
is another extension of $\lambda_w$. Clearly $\lambda_w(\F)\subseteq \F$ as $w\in \F$. So this contradicts the uniqueness of the extension from Theorem \ref{thm:uniqueness_criteria2}(i). Hence $\alpha_t(\tilde{\lambda}(u)) = \tilde{\lambda}(u)$ for all $t \in \mathbb{T}$, and therefore $\tilde{\lambda}(u) \in \BB$.
\end{proof}

The following is motivated by the representation theory of $\Q$ in \cite{LL12}. 

\begin{prop}
Let $\lambda \in \operatorname{End}(\C)$ with $\lambda(s_i)=S_i\in \Q$ ($i=1,2$). If, in the canonical representation of $\Q$, the subspaces $\bigcap\limits_{n=1}^\infty \operatorname{Ran}(S_i^n)$ are finite-dimensional for $i=1,2$, then $\lambda$ has at most one extension to an endomorphism on $\Q$.
\end{prop}

\begin{proof}
It suffices to show that if $\lambda$ is an extendable endomorphism to $\Q$, 
then the extension is unique. For a contradiction, assume that $\lambda$ has two distinct extensions $\tilde{\lambda}_1, \tilde{\lambda}_2 \in \operatorname{End}(\Q)$. Repeated use of  \(
\tilde{\lambda}(u) = S_1 S_2^* + S_2 \tilde{\lambda}(u) S_1^*
\)
gives
\begin{equation}\label{E:wold_decom}
\tilde{\lambda}_1(u) 
 = \sum_{i=1}^{n} S_2^i S_1 S_2^* S_1^{*i} 
   + S_2^{n+1}\tilde{\lambda}_1(u)S_1^{*\, n+1}.
\end{equation}
By \cite[Proposition~4.1]{LL12}, the sum 
\(\sum\limits_{i=1}^{n} S_2^n S_1 S_2^* S_1^{*i}\) strongly converges  to some operator $V$, and there exists
\(
\widetilde{W}_1 = \text{SOT-}\lim\limits_{n\to \infty} S_2^{\,n+1} \tilde{\lambda}_1(u) S_1^{*\,n+1}
\)
such that 
\begin{equation}\label{E:decom_1}
\tilde{\lambda}_1(u) = V + \widetilde{W}_1.  
\end{equation}
Here $\widetilde{W}_1$ is a partial isometry on $\ell^2(\mathbb{Z})$ which restricts to a unitary 
\[
\widetilde{W}_1 : \bigcap_{n=1}^\infty \operatorname{Ran}(S_1^n) \;\longrightarrow\; \bigcap_{n=1}^\infty \operatorname{Ran}(S_2^n),
\]
and vanishes on the orthogonal complements \cite[Proposition 4.1 \& Remark 4.2]{LL12}. Similarly, we have \begin{equation}\label{E:decom_2}
\tilde{\lambda}_2(u) = V + \widetilde{W}_2
\end{equation}
With the finite dimensionality of $\bigcap\limits_{n=1}^\infty \operatorname{Ran}(S_i^n)$, we have that $\widetilde{W}_1$ and $\widetilde{W}_2$ are compact operators. Hence 
\(
\tilde{\lambda}_1(u) - \tilde{\lambda}_2(u) 
   = \widetilde{W}_1 - \widetilde{W}_2 \;\in\; \K(\ell^2(\mathbb{Z})).
\)
This is impossible as $\Q \cap \K(\ell^2(\mathbb{Z})) = \{0\}$.
\end{proof}

The necessary condition below, together with Proposition~\ref{P:Vclassifi}, suggests that uniqueness of extensions is a natural rather than exceptional phenomenon. Thus, non-unique extensions, if they exist, are expected to arise only in special cases.

\begin{prop}\label{prop:uniqueness}
Suppose that $\lambda_w$ is an extendable endomorphism of $\mathcal{O}_2$. If $\lambda_w$ admits a unique extension $\tilde{\lambda}_w$ to $\mathcal{Q}_2$, then
\[
\lambda_w(\mathcal{O}_2)'\cap \mathcal{Q}_2
\subseteq
C^*(\tilde{\lambda}_w(u))'\cap \mathcal{Q}_2 .
\]
\end{prop}
\begin{proof}
Suppose $\lambda_w$ extends to $\tilde{\lambda}_w$. Then, we have
\[
\tilde{\lambda}_w(u)=T_1T_2^*+T_2\tilde{\lambda}_w(u)T_1^*,
\]
where $T_i=ws_i$ ($i=1,2$). 
Let $v\in \lambda_w(\mathcal{O}_2)' \cap \mathcal{Q}_2$ be a unitary. Then
\[
v\tilde{\lambda}_w(u)v^*
= T_1T_2^* + T_2\bigl(v\tilde{\lambda}_w(u)v^*\bigr)T_1^* .
\]
Hence the map $\phi:\mathcal{Q}_2\to\mathcal{Q}_2$ defined by
$\phi(u)=v\tilde{\lambda}_w(u)v^*$ and $\phi(s_i)=T_i$ is an endomorphism extending $\lambda_w$.
Assume $\lambda_w$ has a unique extension. Then $\phi(u)=\tilde{\lambda}_w(u)$, so $v$ commutes with $\tilde{\lambda}_w(u)$. Thus
$v\in C^*(\tilde{\lambda}_w(u))' \cap \mathcal{Q}_2$.
\end{proof}

\subsection{Some applications of the maximality of $\C\subseteq\Q$ to $\operatorname{End}(\Q)$}
%%%%%%%%%%%%%%%%%%%

\label{SS:app}

Combining Theorem \ref{T:endZ} with the maximality of the inclusion $\C\subseteq \Q$ in \cite{OY26} yields several special properties of extendable automorphisms of $\C$.

\begin{theo}
\label{T:auteu}
Suppose that $\lambda_V\in \operatorname{End}(\mathcal{O}_2)$ can be extended to $\lambda\in \operatorname{End}(\mathcal{Q}_2)$. Then 
\begin{enumerate}
\item  $\lambda_V\in \operatorname{Aut}(\mathcal{O}_2) \iff \lambda\in \operatorname{Aut}(\mathcal{Q}_2)$;

\item $\lambda$ is the unique extension of $\lambda_V$
provided that $\lambda_V\in \operatorname{Aut}(\mathcal{O}_2)$. 
\end{enumerate}
\end{theo}

\begin{proof}
(i) ``$\Longrightarrow$'':  If $\lambda_V\in \operatorname{Aut}(\mathcal{O}_2)$, then $\mathcal{O}_2\subseteq \lambda(\mathcal{Q}_2)\subseteq Q_2$. 
But since $\lambda(\mathcal{Q}_2)\cong \mathcal{Q}_2$ and $\mathcal{Q}_2\not\cong \mathcal{O}_2$, one has $ \lambda(\mathcal{Q}_2)=\mathcal{Q}_2$ by the tightness of the inclusion $\C\subseteq \Q$ (\cite[Theorem 3.7]{OY26}). 
Thus $\lambda\in  \operatorname{Aut}(\mathcal{Q}_2)$. 

``$\Longleftarrow$'': To the contrary, assume that $\lambda_V\in \operatorname{End}(\mathcal{O}_2)$ is proper. Then $V\not\in \lambda_V(\mathcal{O}_2)$ by a property of endomorphisms of $\C$. 
So by Theorem \ref{T:endZ}(iii) there is a unitary $\widetilde V\in \mathcal{U}(\Q)\setminus {\mathcal{U}(\C)}$ 
such that $\lambda(\widetilde V)=V$. 
 
Consider the sub-$C^*$-algebra $\mathcal{A}:=C^*(s_1, s_2, \widetilde V)$.
Then $\mathcal{O}_2 \subsetneq \mathcal{A} \subseteq \mathcal{Q}_2$, and hence $\mathcal{A}=\mathcal{Q}_2$ by \cite[Theorem 3.7]{OY26}.
Applying $\lambda$, we obtain
\[
\mathcal{Q}_2
= \lambda(\mathcal{Q}_2)
= \lambda(\mathcal{A})
= C^*\big(\lambda(s_1), \lambda(s_2), \lambda(\widetilde V)\big)
= C^*\big(\lambda(s_1), \lambda(s_2), V\big).
\]
Since $\lambda(s_1), \lambda(s_2), V$ all belong to $\mathcal{O}_2$, it follows that
$
\mathcal{Q}_2 \subseteq \mathcal{O}_2
$,
a contradiction. Therefore, $\lambda_V$ is an automorphism of $\mathcal{O}_2$.

(ii)
By Proposition \ref{P:Vclassifi}, it suffices to show that Eq.~\eqref{E:u1*u2} has only the trivial solution. Let $W$ be a solution of Eq.~\eqref{E:u1*u2}. 
Then $Vs_1W=WVs_1$. Since $C^*(s_1)'\cap \mathcal{Q}_2=\mathbb CI$ (\cite{ACR21}) and $\lambda$ is an automorphism as shown in (i), we obtain 
$C^*(Vs_1)'\cap \mathcal{Q}_2=C^*(\lambda(s_1))'\cap \lambda(\mathcal{Q}_2)=\lambda(C^*(s_1)'\cap \mathcal{Q}_2)=\mathbb C I$. 
Thus $W=\alpha I$ for some $\alpha \in \mathbb T$. Then substituting it to Eq.~\eqref{E:u1*u2} yields $\alpha s_2s_2^*=s_2s_2^*$. Hence $\alpha=1$. 
\end{proof}

Notice that Theorem \ref{T:auteu} (ii) is proved in \cite{ACR18} in a very different way.

%%%%%%%%%%%%%%%%%%%%

\section{Diagonal automorphisms of \(\Q\)}
\label{S:diagaut}

Our main goal in this section is to explicitly construct an automorphism of \( \Q \) that, to the best of our knowledge, is not a composition of the extendable automorphisms of \( \C \) currently known in the literature. This automorphism is further used to show that 
the canonical image of $\operatorname{Aut}(\Q, \C)$ in $\operatorname{Out}(\Q)$ is non-abelian. Accordindly, we resolve several open questions on extendable diagonal automorphisms of \( \C \) posed in \cite{ACR21}.

To achieve our goal, we first need to construct a special real continuous function on $\mathbf{Z}_2$.
  
%%%%%%%%%%%%%%%%%%%%%
\begin{theo}\label{T:cont_fn}
There exists a continuous function $\phi:\mathbf{Z}_2 \to \mathbb{R}$ such that
  \begin{enumerate}
    \item $\phi(k) = \phi(2k) + \phi(2k-1)$ for all $k \in \z$,
    \item $\phi(\mathbf{Z}_2)$ is not contained in $\mathbb{Z}$.
\end{enumerate}
\end{theo}

\begin{proof}
For all non-negative integer $n\ge 0$, let 
\begin{align*}
P_0:=1,\ P_n &\coloneq \prod_{j=0}^{n-1} \left(1+  e^{-\pi i/2^{j+1}}\right)\quad \text{for all }n\ge 1.
\end{align*}
Define $f:\mathbf{Z}_2 \to \mathbb{C}$ by
\begin{equation}
\label{eqn:funct_defn}
 f(k) = \sum_{n=0}^{\infty} \frac{1}{P_n} \chi(2^{-(n+1)} k) \quad\text{for all }k\in\z
\end{equation}
(refer to Section \ref{S:Pre} for the notation $\chi$).

To show that $f$ is well-defined and continuous, we first prove that $\sum\limits_{n=0}^\infty \frac{1}{\vert P_n\vert} <\infty$.
Using $|1 + e^{-\pi i/2^{j+1}}| = 2 \left|\cos\frac{\pi}{2^{j+2}}\right|$, we see that
\[
\frac{1}{|P_n|}
= \frac{1}{\prod_{j=0}^{n-1} |1 + e^{-\pi i/2^{j+1}}|}
= \frac{1}{2^n \prod_{j=0}^{n-1} \left|\cos \frac{\pi}{2^{j+2}}\right|}.
\]
Since $|\cos(\pi/2^{j+2})| \le 1$, we get the bound
\[
\frac{1}{|P_n|} \le \frac{1}{2^n} \prod_{j=0}^{\infty} \frac{1}{\left|\cos\frac{\pi}{2^{j+2}}\right|}.
\]
Using Viet\`e's formula
$\sin \theta = \theta \prod\limits_{r=0}^{\infty} \cos \frac{\theta}{2^{r+1}}$ with $\theta = \frac{\pi}{2}$, we obtain $1 = \frac{\pi}{2} \prod\limits_{r=0}^{\infty} \cos \frac{\pi}{2^{r+2}}$. Hence
\[
\prod_{r=0}^{\infty} \left|\cos \frac{\pi}{2^{r+2}}\right| = \frac{2}{\pi}\qquad \text{and}\qquad
\prod_{j=0}^{\infty} \frac{1}{\left|\cos(\pi/2^{j+2})\right|} = \frac{\pi}{2}.
\]
Thus
\[
\sum_{n=0}^{\infty} \frac{1}{|P_n|} \le \frac{\pi}{2}\sum_{n=0}^{\infty} \frac{1}{2^n} = \pi < \infty.
\]

Notice that $\vert \frac{1}{P_n} \chi(2^{-(n+1)} k)\vert \le \frac{1}{|P_n|}$ for all $k\in \z$. Clearly, every term in the series of \eqref{eqn:funct_defn} is continuous. So, by the Weierstrass M-test (see \cite[Theorem 7.10 \& 7.12]{Rudin76} for example), we get that $f$ is well-defined and continuous.

We now check the functional equation $f(k) = f(2k) + f(2k-1)$ for $k\in \mathbb{Z}$.  
For each $n \ge 1$, using 
the identity $e^{-\pi i/2^n} = \chi(-2^{-(n+1)})$, one has 
\[
\chi(2^{-n} k)\, e^{-\pi i/2^n} = \chi(2^{-n} k)\, \chi(-2^{-(n+1)}) 
=\chi(2^{-(n+1)}(2k-1)).
\]
Combining this equality with $P_n = P_{n-1} \bigl(1 + e^{-\pi i/2^n}\bigr)$ gives
\[
\chi(2^{-n} k) + \chi(2^{-(n+1)}(2k-1)) = \chi(2^{-n} k)\, (1 + e^{-\pi i/2^n}) = \chi(2^{-n} k)\, \frac{P_n}{P_{n-1}}.
\]
Therefore,
\[
\frac{\chi(2^{-n} k) + \chi(2^{-(n+1)}(2k-1))}{P_n} = \frac{\chi(2^{-n} k)}{P_{n-1}}, \qquad \forall\, n \ge 1.
\]
Since $\chi(k)=1$ and $\chi(\frac{2k-1}{2})=-1$ for all $k\in \mathbb{Z}$, we get
\begin{equation}\label{eqn:funct_eq1}
f(2k) + f(2k-1) = \sum_{n=1}^{\infty} \frac{\chi(2^{-n} k)}{P_{n-1}} = \sum_{n=0}^{\infty} \frac{\chi(2^{-(n+1)} k)}{P_n} = f(k)
\end{equation}
for all $k\in \mathbb{Z}$. Now using the density of $\mathbb{Z}$ in $\z$ and continuity of $f$, 
%and \eqref{eqn:funct_eq1}, we obtain $f(k)=f(2k) + f(2k-1)$ 
we obtain \eqref{eqn:funct_eq1} holds true for all $k\in\z$. Notice that the 
imaginary part $\Im(f)$ is a continuous function $\mathbf{Z}_2 \to \mathbb{R}$ 
 satisfying the same functional equation.

Let $\phi:=\Im(f)$, the imaginary part of $f$. We now show that $\phi(\mathbf{Z}_2)$ is not contained in $\mathbb{Z}$. For this, it suffices to show that $\Im(f(0))=\Im\left(\sum_{n=0}^{\infty}\frac{1}{P_n}\right)\notin\mathbb{Z}$. One can first check that
\[
P_n=\frac{2}{1-e^{-\pi i/2^n}}
\qquad \text{for all } n\geq 0,
\]
implying
\begin{equation}\label{E:exp1}
 \frac{1}{P_n}
=\frac{\left(1-e^{-i\pi/2^n}\right)}{2}.   
\end{equation}
Then using Euler's formula, we obtain
\[
\sum_{n=0}^{\infty}\frac{1}{P_n}
=
\sum_{n=0}^{\infty}
\sin^2\left(\frac{\pi}{2^{n+1}}\right)
+
i\sum_{n=1}^{\infty}
\frac{\sin(\pi/2^n)}{2}.
\]
Hence,
\[
\Im\left(\sum_{n=0}^{\infty}\frac{1}{P_n}\right)
=
\sum_{n=1}^{\infty}\frac{\sin(\pi/2^n)}{2}.
\]
Observe that 
\begin{align*}
1&<\frac{\sin(\pi/2)}{2}
+\frac{\sin(\pi/4)}{2}
+\frac{\sin(\pi/8)}{2}
+\frac{\sin(\pi/16)}{2}  \\  
&<\sum_{n=1}^{\infty}\frac{\sin(\pi/2^n)}{2}
<\sum_{n=1}^{\infty}\frac{\pi}{2^{n+1}}
=
\frac{\pi}{2}
<2,
\end{align*}
where the third ``$<$" holds as $\sin(x)<x$ for $x>0$. 

So,
\[
1<
\Im\left(\sum_{n=0}^{\infty}\frac{1}{P_n}\right)
<2.
\]
It clearly follows that $\Im(f(0))\notin\mathbb{Z}$.
\end{proof}

We are now in a position to construct an automorphism of \( \Q \) with 
desired properties.

\begin{theo}
\label{T:lambdad}
There exists a unitary \( d \in \D \) such that the corresponding automorphism \( \lambda_d \in \operatorname{Aut}(\C) \) satisfies:
\begin{enumerate}
\item \( \lambda_d \) extends to an automorphism of \( \Q \);
\item \( \lambda_d(x) = x \) for all \( x \in C^*(\D, s_2) \);
\item \( \lambda_d \) does not lie in the subgroup of \( \operatorname{Aut}(\C) \) generated by the flip-flop \( \sigma \), 
the gauge automorphisms \( \alpha_t \), and inner automorphisms.
\end{enumerate}
\end{theo}

\begin{proof}
Let \( \phi: \mathbf{Z}_2 \to \mathbb{R} \) be the continuous function from Theorem \ref{T:cont_fn}. Define \( \tilde{d}: \mathbf{Z}_2 \to \mathbb{T} \) by
\begin{align}
\label{E:tilded}
\tilde{d}(k) = e^{2\pi i \phi(k)}\quad \text{for all}\quad k\in \mathbf{Z}_2.
\end{align}
Then \( \tilde{d} \) is continuous and satisfies
\begin{equation}\label{E:funct_eqn1}
    \tilde{d}(k) = \tilde{d}(2k)\tilde{d}(2k-1) \quad \text{for all}\quad k \in \mathbb{Z}.
\end{equation}
Identify \( \tilde{d} \) with a unitary in \( \D \) and set
\[
d = s_2 s_2^* + \tilde{d}\, s_1 s_1^*.
\]
Note that $\lambda_d \in \operatorname{Aut}(\C)$ since $d\in \D$. In the canonical representation of \( \Q \), one has 
\[
(s_1 s_2^* + s_2 \tilde{d} u s_1^*) e_k =
\begin{cases}
e_{k+1} & \text{if } k \text{ is even}, \\
\tilde{d}\!\left(\frac{k+1}{2}\right) e_{k+1} & \text{if } k \text{ is odd},
\end{cases}
\]
and
\[
(d^* \tilde{d} u d) e_k =
\begin{cases}
e_{k+1} & \text{if } k \text{ is even}, \\
\tilde{d}(k)\tilde{d}(k+1)\, e_{k+1} & \text{if } k \text{ is odd}.
\end{cases}
\]
Using the relations satisfied by \( \tilde{d} \), we see that  these expressions agree on the standard basis \( \{e_k : k \in \mathbb{Z}\} \) of $\ell^2(\mathbb Z)$. Hence
\[
d^* \tilde{d} u d = s_1 s_2^* + s_2 \tilde{d} u s_1^*.
\]
By Theorem \ref{T:endZ}, it follows that \( \lambda_d \) extends to an automorphism $\tilde{\lambda}_d$ of \( \Q \) with $\tilde{\lambda}_d (u)=\tilde{d}u$.

Since \( d \in \D \), we have \( \lambda_d(x) = x \) for all \( x \in \D \). Moreover,
\[
\lambda_d(s_2) = d s_2 = s_2,
\]
so \( \lambda_d(x) = x \) for all \( x \in C^*(\D, s_2) \).

Next, we show that \( \lambda_d \) does not belong to the subgroup of \( \operatorname{Aut}(\C) \) generated by the flip-flop \( \sigma \), 
the gauge automorphisms \( \alpha_t \), and  inner automorphisms.

Note that \( \lambda_d \) is outer by \cite[Corollary B]{MT93}, as it fixes \( s_2 \). Moreover, it is neither \( \sigma \) nor  \( \alpha_t \) for \( t \ne 1 \), since $\alpha_t$ ($t\ne 1$) does not fix \( s_2 \). Also, for any unitary \( w \in \mathcal{U}(\C) \) and $t\in \mathbb T$ one has 
$\sigma\circ \alpha_t=\alpha_t \circ \sigma$, and 
\begin{align*}
\sigma \circ \operatorname{Ad}(w) = \operatorname{Ad}(\sigma(w)) \circ \sigma,\quad
\alpha_t \circ \operatorname{Ad}(w) = \operatorname{Ad}(\alpha_t(w)) \circ \alpha_t.
\end{align*}
Thus it suffices to show that \( \lambda_d \) cannot be written in the form
\begin{align}
\label{E:no}
\lambda_d = \operatorname{Ad}(w) \circ \alpha_t 
\qquad \text{or} \qquad 
\lambda_d = \operatorname{Ad}(w) \circ \alpha_t \circ \sigma
\end{align}
for any \( t \in \mathbb{T} \) and \( w \in \mathcal{U}(\C) \).

To the contrary, first assume that \( \lambda_d = \operatorname{Ad}(w) \circ \alpha_t \). For \( x \in \D \), we have
\[
w x w^* = x.
\]
So \( w \in \D \) as \( \D \) is maximal abelian. Then
\[
t\, w s_2 w^*=\operatorname{Ad}(w) \circ \alpha_t (s_2)  =\lambda_d(s_2)=s_2.
\]
Evaluating on the basis vector \( e_0 \) and using the fact that \( w \in \D \), we obtain \( t = 1 \). Hence \( \lambda_d = \operatorname{Ad}(w) \) is inner, a contradiction.

Next suppose that \( \lambda_d = \operatorname{Ad}(w) \circ \alpha_t \circ \sigma \). Recall that if \( x \in \D \), then $\alpha_t(x)=x$ and $\sigma(x)\in\D$. So $\lambda_d(\sigma(x))=\sigma(x)$ by (ii) proved above. Then, for all $x\in 
\D$, we have 
\[
w x w^*= \operatorname{Ad}(w) \circ \alpha_t \circ \sigma(\sigma(x)) 
= \lambda_d(\sigma(x)) =\sigma(x).
\]
Thus \( w \) normalizes \( \D \).
By  \cite[Theorem 4.8]{ACR20}, we can write \( w = q P \), where \( q \in \D \) is unitary and \( P \in V_2\), the Thompson group in  $\mathcal{U}(\C)$.

Notice that simple calculations yield
\[
\operatorname{Ad}(w) \circ \alpha_t \circ \sigma 
= \lambda_{t\, w (s_1 s_2^* + s_2 s_1^*) \varphi(w)^*},
\]
%(also refer to Theorem \ref{T:endZ}),
where $\varphi$ is the canonical shift on $\C$ (refer to Section \ref{S:Pre}).
Hence it follows from
$\lambda_d 
= \operatorname{Ad}(w) \circ \alpha_t \circ \sigma$
that 
\[
d = t\, w (s_1 s_2^* + s_2 s_1^*) \varphi(w)^*
\]
by the Cuntz-Takesaki correspondence.
Thus 
\[
d = t\, qP (s_1 s_2^* + s_2 s_1^*) \varphi(qP)^*\implies
P (s_1 s_2^* + s_2 s_1^*) \varphi(P)^* 
= \overline{t}\, q^* d \varphi(q) \in \D.
\]
Clearly, one also has \( P (s_1 s_2^* + s_2 s_1^*) \varphi(P)^* \in V_2 \). 
Hence we must have
\[
P (s_1 s_2^* + s_2 s_1^*) \varphi(P)^* = 1,
\]
which implies 
\[
s_1 s_2^* + s_2 s_1^* = P^* \varphi(P).
\]
Then the 1-1 correspondence (Theorem \ref{T:endZ}) yields
\[
\sigma = \lambda_{s_1 s_2^* + s_2 s_1^*}=\operatorname{Ad}(P^*),
\]
implying that $\sigma$ would be inner. This is impossible as it is well-known that $\sigma$ is outer.

Thus far, we have known no \( t \in \mathbb{T} \) and \( w \in \mathcal{U}(\C) \) satisfy either of the identities in \eqref{E:no}. Therefore \( \lambda_d \) does not lie in the subgroup generated by \( \sigma \), the gauge automorphisms, and inner automorphisms.
\end{proof}

In \cite{ACR18}, Aiello, Conti, and Rossi showed that the outer automorphism group $\operatorname{Out}(\Q)$ of $\Q$ is non-abelian. As a consequence of Theorem~\ref{T:lambdad}, in what follows, we show more: The canonical image of $\operatorname{Aut}(\Q, \C)$ in $\operatorname{Out}(\Q)$ is also non-abelian. This yields a positive answer to \cite[Question 9.25]{ACR21}.

\begin{theo}
\label{T:AutQOnonabe}
The canonical image of $\operatorname{Aut}(\Q, \C)$ in $\operatorname{Out}(\Q)$ is non-abelian.
\end{theo}
\begin{proof}
Let $\tilde{\lambda}_d$ be the unique extension of $\lambda_d$ constructed in Theorem~\ref{T:lambdad}. First notice that both $\tilde{\lambda}_d$ and $\sigma$ are outer automorphisms. So, to prove the theorem, it suffices to show that there is no unitary $w\in \Q$ such that
\(
\sigma\circ\tilde{\lambda}_d
=
\operatorname{Ad}(w)\circ\tilde{\lambda}_d\circ\sigma.
\)

Suppose such a unitary $w$ exists by contradiction. Then
\begin{equation}\label{E:non_a1}
s_1
=\sigma\circ\tilde{\lambda}_d(s_2)
=\operatorname{Ad}(w)\circ\tilde{\lambda}_d\circ\sigma(s_2)
=wds_1w^*.
\end{equation}
Since $\lambda_d$ is extendable, we have
\[
s_2\tilde{d}u=\tilde{d}u\,ds_1,
\]
which gives
\[
ds_1=(\tilde{d}u)^*s_2(\tilde{d}u).
\]
Combining this with $s_1=u^*s_2u$, equation~\eqref{E:non_a1} implies
\[
s_2uw(\tilde{d}u)^*
=
uw(\tilde{d}u)^*s_2.
\]
By \cite[Theorem~3.20]{ACR18}, we obtain
\(
uw(\tilde{d}u)^*=\mu1
\)
for some $\mu\in\mathbb{T}$. Therefore,
\[\
w=\mu u^*\tilde{d}u.
\]
In particular, one has $w\in \D$.

%%%%%%%%%%%
Notice that $\sigma\circ\tilde{\lambda}_d\circ\sigma\circ\tilde{\lambda}_{d^*}(\C)=\C$, and that its restriction to $\C$ is $\sigma\circ\lambda_d\circ\sigma\circ\lambda_{d^*}$. On the one hand, by \eqref{E:non_a1}, we have
\[
\sigma\circ\lambda_d\circ\sigma\circ\lambda_{d^*}
=
\operatorname{Ad}(w).
\]
On the other hand,
\[
\sigma\circ\lambda_d\circ\sigma\circ\lambda_{d^*}
=
\lambda_W,
\]
where
\[
W=\sigma(d)s_2s_2^*+d^*s_1s_1^*.
\]
Indeed,
\[
\begin{aligned}
W
&=
\sigma\circ\lambda_d\circ\sigma\circ\lambda_{d^*}(s_2)s_2^*
+
\sigma\circ\lambda_d\circ\sigma\circ\lambda_{d^*}(s_1)s_1^*\\
&=
\sigma(d)s_2s_2^*+d^*s_1s_1^*.
\end{aligned}
\]
Thus
$\lambda_W = \operatorname{Ad}(w)$ implies 
$\lambda_W=\lambda_{w\varphi(w)^*}$. So
\[
W=\sigma(d)s_2s_2^*+d^*s_1s_1^*=w\varphi(w)^*.
\]
It is clear that $\sigma(W)=W^*$.
Hence, combining this with $w, \varphi(w)\in \D$, we obtain 
\[
w^*\varphi(w)
=
W^*
=
\sigma(W)
=
\sigma(w)\varphi(\sigma(w)^*).
\]
It follows that
\[
\varphi(w\sigma(w))=w\sigma(w).
\]
By \cite[Theorem~3.20]{ACR18}, we obtain
\(
w\sigma(w)=\delta1
\)
for some $\delta\in\mathbb{T}$. Consequently,
\begin{equation}\label{E:non_a2}
\mu^2 u^*\tilde{d}u^2\sigma(\tilde{d})u^*
=
w\sigma(w)
=
\delta1.
\end{equation}
Observe that $\sigma(\tilde{d})\in\mathcal{D}_2$. In the canonical representation of $\Q$, one has
\[
\sigma(\tilde{d})e_k=\tilde{d}(-1-k)e_k \quad \text{for all }k\in \mathbb Z.
\]
Applying this to  \eqref{E:non_a2}, we get
\begin{equation}\label{E:non_a3}
    \mu^2 \tilde{d}(k+1) \tilde{d}(-k) = \delta   \quad \text{for all } k\in \mathbb{Z}.
\end{equation}
From \eqref{E:funct_eqn1}, one can easily see that $\tilde{d}(-1)=\tilde{d}(2)=1$. Then letting $k=1$ in \eqref{E:non_a3} gives $\mu^2=\delta$. This implies $\tilde{d}(k+1)\tilde{d}(-k)=1$. In particular, $\tilde{d}(0)\tilde{d}(1)=1$. Recall from \eqref{E:tilded} that
\[
\tilde{d}(k)=e^{2\pi i\phi(k)}\quad \text{for all } k\in\mathbf{Z}_2,
\]
where $\phi:\mathbf{Z}_2\to\mathbb{R}$ is the continuous function from Theorem~\ref{T:cont_fn}. Thus, we must have $\phi(0)+\phi(1)\in\mathbb{Z}$. 

To show this is impossible, by the definition of $\phi$ given in the proof of Theorem~\ref{T:cont_fn}, it suffices to prove that the imaginary part of $f(0)+f(1)$ is not an integer, where $f$ is as in \eqref{eqn:funct_defn}. From \eqref{E:exp1}, we obtain
\begin{align*}
f(0)+f(1)
&=\sum_{n=0}^{\infty}\frac{1+e^{i\pi/2^n}}{P_n} =\sum_{n=1}^{\infty}\frac{(1+e^{i\pi/2^n})(1-e^{-i\pi/2^n})}{2}\\
&=i\sum_{n=1}^{\infty}\sin(\pi/2^n).
\end{align*}
As before, we have the following estimations
\begin{align*}
2&<
\sin(\pi/2)+\sin(\pi/4)+\sin(\pi/8)\\
&<\sum_{n=1}^{\infty}\sin(\pi/2^n)
\le 1+\sum_{n=2}^{\infty}\frac{\pi}{2^n}
=1+\frac{\pi}{2}<3.
\end{align*}
Hence, $2<\Im(f(0)+f(1))<3$. So $\phi(0)+\phi(1)\notin\mathbb{Z}$, a contradiction. 
Therefore no such unitary $w$ exists.
\end{proof}

%%%%%%%%%%Added on April 16
\section{Automorphisms of $\Q$ preserving $C^*(u)$ and the related Weyl group $\mathcal{W}(\Q, C^*(u))$}
\label{S:Weyl}

In this section, we first show that $C^*(u)$ is a Cartan subalgebra of $\Q$. This naturally associates $\Q$ a Deaconu--Renault groupoid $\Gamma(\mathbb{T},\gamma)$, which provides additional tools for studying $\Q$. We then completely describe $\operatorname{Aut}(\Q,C^*(u))$, the group of automorphisms of $\Q$ that preserve $C^*(u)$, and obtain a full description of the corresponding Weyl group $\mathcal{W}(\Q,C^*(u))$.
%%%%%%
Along the way, we also characterize the group of inner automorphisms preserving $C^*(u)$, the group of \textit{unitary} normalizers of $C^*(u)$ in $\Q$, and the topological full group of $\Gamma(\mathbb{T},\gamma)$.

\subsection{$C^*(u)$ is Cartan in $\Q$}
\label{SS:uCartan}
In this short subsection, making use of the various realizations of $\Q$ appearing in the literature, we show that $C^*(u)$ is not merely a MASA (\cite{ACR18}), but is in fact a Cartan subalgebra of $\Q$ in the sense of \cite[Definition 5.1]{Ren08}. This fact plays a vital role throughout this section.

Unlike the notion of regularity used by Aiello, Conti, and Rossi in \cite[Proposition 3.1]{ACR20}, where normalizers are required to be partial isometries, no such restriction is imposed here.
\begin{prop}\label{prop:cartan2}
The $C^*$-subalgebra $C^*(u)$ is Cartan in $\Q$.
\end{prop}

\begin{proof}
Consider the covering map $\gamma : \mathbb{T}\to \mathbb{T}$ given by $\gamma(z)=z^2$, and the $C^*$-algebra $C^*(\mathbb{T}, \gamma)$ associated with the Deaconu--Renault system $\Gamma(\mathbb T, \gamma)$ \cite[Example 3]{Deaconu95}.

By \cite[Theorem 9.1]{EV06}, $C^*(\mathbb{T}, \gamma)$ is isomorphic to the Exel crossed product $C(\mathbb{T}) \rtimes_{\alpha, L} \mathbb{N}$ via an isomorphism that restricts to the identity on $C(\mathbb{T})$, where $\alpha : C(\mathbb{T})\to C(\mathbb{T})$ is an endomorphism, and $L: C(\mathbb{T}) \to C(\mathbb{T})$ is a transfer operator, which are defined by
\[
\alpha(f)(z)=f(z^2)\quad\text{and}\quad
L(f)(z) = \frac{1}{2}\sum_{w^2 = z} f(w),
\]
respectively. 

Larsen and Li \cite[Page 1398]{LL12} identify $C(\mathbb{T}) \rtimes_{\alpha, L} \mathbb{N}$ with $\Q$, and under this identification the subalgebra $C(\mathbb{T})$ corresponds to $C^*(u)$ (see also \cite[Proposition 3.3]{LRR11}).

It is known that $\gamma$ is topologically free \cite[Example 3]{Deaconu95}, and hence $C(\mathbb{T})$ is a Cartan $C^*$-subalgebra of $C^*(\mathbb{T}, \gamma)$ (\cite{Ren08}). Transporting this structure through the above isomorphisms, it follows that $C^*(u)$ is a Cartan subalgebra of $\Q$.
\end{proof}

In the rest of this section, we shall see that the Deaconu--Renault groupoid model $\Gamma(\mathbb T, \gamma)$ in Proposition \ref{prop:cartan2} provides useful tools to facilitate studying $\Q$.

\subsection{The structure of $\operatorname{Aut}(\Q, C^*(u))$}
In this section, we use the Deaconu–Renault groupoid model $\Gamma(\mathbb{T}, \gamma)$ in Proposition \ref{prop:cartan2} to study $\operatorname{Aut}(\Q, C^*(u))$, the group of all automorphisms of $\Q$ that preserve $C^*(u)$.

To simplify our writing, let $G\coloneq \Gamma(\mathbb{T}, \gamma)$ in what follows.

\begin{prop}\label{prop:orbit_equiv}
Let $G$ be the Deaconu–Renault groupoid associated with $\gamma : \mathbb{T} \to \mathbb{T}$ given by $\gamma(z) = z^2$. Suppose that 
\begin{enumerate}
\item $l, k : \mathbb{T} \to \mathbb{N}$ are continuous maps, and
\item $h : \mathbb{T} \to \mathbb{T}$ is a homeomorphism,
\end{enumerate}
such that 
\begin{align}\label{eqn:orbit_equiva}
\gamma^{l(z)}(h(z)) = \gamma^{k(z)}\big(h(\gamma(z))\big)
\qquad \text{for all } z \in \mathbb{T}.
\end{align}
holds for all $z \in \mathbb{T}$. Then $l$ and $k$ are constant maps with $l = k + 1$, and $h$ is of the form
\[
h(z) = \xi\, z^{\pm 1},
\]
where $\xi\in \mathbb{Z}(2^\infty)$. 
\end{prop}

\begin{proof}
Since $\mathbb{T}$ is connected and $\mathbb{N}$ is discrete, both $l$ and $k$ are constant. By abusing the notation, we still denote these constants by $l$ and $k$. Given a homeomorphism $h$ of $\mathbb{T}$, there exists a continuous function $\phi : \mathbb{R} \to \mathbb{R}$ such that
\[
h(e^{2\pi i t}) = e^{2\pi i \phi(t)}
\]
with
\[
\phi(t+m)=\phi(t)\pm m \quad \text{for all } t\in \mathbb{R}, \ m\in\mathbb{Z},
\]
where the sign corresponds to orientation \cite[Chapter 2]{URS22}.  

Notice that the relation \eqref{eqn:orbit_equiva} becomes
\[
e^{2\pi i 2^{l}\phi(t)} = e^{2\pi i 2^{k}\phi(2t)}.
\]
Hence \(
2^{l}\phi(t) - 2^{k}\phi(2t) \in \mathbb{Z} \).
By continuity, there is $c\in \mathbb Z$ such that 
%this integer-valued expression is constant, say
\begin{align}
\label{E:phic}
2^{l}\phi(t) - 2^{k}\phi(2t) = c \quad \text{for all } t\in \mathbb{R}.
\end{align}
Note that
\begin{equation}\label{E:phi0}
    c= 2^l \phi(0)- 2^k \phi(0).
\end{equation}
%%%%%%

\textsf{Case 1:} \underline{$h$ preserves the orientation.}
Then  
\begin{align}
\label{E:po}
\phi(t+m)=\phi(t)+m \quad \text{for all }t\in \mathbb{R},\  m\in\mathbb{Z}.
\end{align}

On one hand, combing this with \eqref{E:phic} and \eqref{E:phi0} yields
\[
c = 2^{l}(\phi(0)+m) - 2^{k}(\phi(0)+2m)
= c + (2^{l}-2^{k+1})m\quad\text{for all }m\in \mathbb{Z}.
\]
This implies $l = k + 1$. Consequently, we obtain $c= 2^k \phi(0)$ from \eqref{E:phi0}.

Substituting $c$ and $l=k+1$ into \eqref{E:phic} and simplifying, we obtain
\begin{equation}\label{E:phisim}
 2\phi(t)-\phi(2t)=\phi(0)\quad \text{for all }t\in \mathbb{R}.   
\end{equation}
Let 
\[
\psi(t):=\phi(t)-\phi(0)\quad\text{for all }t\in\mathbb{R}.
\]
Then it is clear from \eqref{E:po} and \eqref{E:phisim} that 
\[
\psi(m)=m\quad\text{and}\quad
\psi(2t)=2\psi(t)\quad \text{for all }m\in\mathbb{Z},\ t\in\mathbb{R}.
\]
Simple induction shows 
\[
\psi(2^nt)=2^n\psi(t)\quad\text{for all }t\in\mathbb{R}\quad \text{and} \quad n\in \mathbb{N}.
\]
Setting $t=p/2^n$ with $p\in \mathbb{Z}$ induces 
\[
\psi(\frac{p}{2^n})=2^{-n}\psi(p)=\frac{p}{2^n} \quad\text{for all }p\in\mathbb{Z}. 
\]
%as $\psi(p)=p$ for all $p\in\mathbb{Z}$.
%%%%%%%%%%%%%
Thus $\psi(t)=t$ for all $t\in\mathbb{Z}[1/2]$. By density of $\mathbb{Z}[1/2]$ in $\mathbb{R}$ and continuity of $\psi$, it follows that $\psi(t)=t$ for all $t\in\mathbb{R}$. Thus $\phi(t)=t+\phi(0)$, and hence
\[
h(z)=e^{2\pi i \phi(0)} z\quad \text{for all }z\in \mathbb{T}.
\]

\textsf{Case 2:} \underline{$h$ reverses the orientation.}
In this case, one has 
\begin{align*}
\label{E:ro}
\phi(t+m)=\phi(t)-m \quad \text{for all }t\in \mathbb{R},\  m\in\mathbb{Z}. 
\end{align*}
Similar to the proof of Case 1, one obtains $h(z)=e^{2\pi i \phi(0)} \overline{z}$. 

Therefore, in general
\[
h(z)=\xi z^{\pm 1} \quad \text{with}\quad\xi=e^{2\pi i \phi(0)}
\]
for all $z\in\mathbb{T}$.
Since $2^{k}\phi(0)(=c)\in\mathbb{Z}$ as shown above, we have $\xi^{2^{k}}=1$. So $\xi\in\mathbb{Z}(2^\infty)$, and this completes the proof.
\end{proof}

Recall from Remark \ref{R:u-uni} that for $V\in \mathcal{U}(C^*(u))$, $\beta_V$ is an automorphism in $\operatorname{Aut}_{C^*(u)}(\Q)$  given by $\beta_V=\lambda_{(u,V)}$. 

For $z \in \mathbb{T}$, define a unitary operator on $\ell_2(\mathbb{Z})$ by
\[
U_z e_k = z^k e_k\quad \text{for all }k\in \mathbb{Z}.
\]
It is known that $U_z$ can be identified with an element of $\mathcal{D}_2$ under the canonical representation of $\Q$, if and only if  
$z\in \mathbb{Z}(2^\infty)$ (\cite[Proposition 6.20]{ACR18}).

Below is the main result of this subsection.

\begin{theo}
\label{T:autp}
Keep the same notation as above. Every automorphism $\phi\in \operatorname{Aut}(\Q, C^*(u))$
is of the form
\[
\operatorname{Ad}(U_\xi)\circ \beta_V\quad \text{or} \quad \sigma \circ \operatorname{Ad}(U_\xi)\circ \beta_V
\]
for some $V \in \mathcal{U}(C^*(u))$ and $\xi \in \mathbb{Z}(2^\infty)$. 
\end{theo}
\begin{rmk}
one can easily check that $\operatorname{Ad}(U_\xi)\circ \beta_V$ and $\sigma \circ \operatorname{Ad}(U_\xi)\circ \beta_V$ are automorphisms of $\Q$ mapping $u$ to $\xi u$ and $\xi u^*$, respectively. As a matter of fact, by Theorem \ref{P:u-uni}, one can classify these two classes of automorphisms.  
\end{rmk}

\begin{proof}[Proof of Theorem \ref{T:autp}]
Let $G$ be the Deaconu–Renault groupoid associated with $\gamma : \mathbb{T} \to \mathbb{T}$ given by $\gamma(z) = z^2$. Let $\Phi \in \operatorname{Aut}(G)$, and $\varphi_\Phi$ be the corresponding automorphism in $\operatorname{Aut}(\Q, C^*(u))$ via \cite[Theorem 1.5.7]{Komura25}.\footnote{$\operatorname{Aut}(G)$ with $c$ being the trivial cocycle of $G$, $\varphi_\Phi(f)(g):=f(\Phi^{-1}(g)$ for all $g\in G$.}
Then \cite[Proposition 2.1.3]{Komura25}
\[
\operatorname{Aut}(\Q, C^*(u)) / \operatorname{Aut}_{C^*(u)}(\Q)
= \{[\varphi_\Phi] : \Phi \in \operatorname{Aut}(G)\}.
\]

By \cite[Corollary 4.3.7]{Komura25} and the definition of $\varphi_\Phi$ in \cite[Definition 1.5.5]{Komura25},
we have
\[
\varphi_\Phi(u) = h(u),
\]
where $h : \mathbb{T} \to \mathbb{T}$ is a homeomorphism satisfying the conditions in Proposition \ref{prop:orbit_equiv}. By the classification of such $h$, it follows that
\[
h(z) = \xi\, z^{\pm 1}
\]
for some $\xi\in \mathbb{Z}(2^\infty)$.

From the definitions of $\beta_V$, $\operatorname{Ad}(U_\xi)$ and $\sigma$, we see that there exists an automorphism 
\[
\theta \in \{\operatorname{Ad}(U_\xi)\circ \beta_V,\ \sigma \circ \operatorname{Ad}(U_\xi)\}
\]
such that $\theta(u) = \varphi_\Phi(u)$. Hence
\[
\theta^{-1} \circ \varphi_\Phi(u) = u,
\]
so $\theta^{-1} \circ \varphi_\Phi \in \operatorname{Aut}_{C^*(u)}(\Q)$. Therefore, by Corollary \ref{C:u-uni} one has 
\[
\theta^{-1} \circ \varphi_\Phi = \beta_V
\quad \text{for some } V\in\U(C^*(u)),
\]
and thus $\varphi_\Phi = \theta \circ \beta_V$.

Now let $\phi \in \operatorname{Aut}(\Q)$ with $\phi(C^*(u)) = C^*(u)$. Then $[\phi] = [\varphi_\Phi]$ for some $\Phi \in \operatorname{Aut}(G)$, so $\phi = \varphi_\Phi \circ \beta_{\widetilde V}$ for some $\widetilde V \in \U(C^*(u))$. It follows that $\phi$ is of the stated form.
\end{proof}

\subsection{$\operatorname{Inn}(\Q, C^*(u))$ and $\mathcal{N}_{C^*(u)}(\Q)$}
In this subsection, we apply Theorem~\ref{T:autp} to determine $\operatorname{Inn}(\Q,C^*(u))$, the group of inner automorphisms preserving $C^*(u)$, and $\mathcal{N}_{C^*(u)}(\Q)$, the group of unitary normalizers of $C^*(u)$ in $\Q$. We begin by exhibiting a large class of outer automorphisms of $\Q$, extending the class constructed in \cite[Theorem~5.9]{ACR18}.
\begin{theo}\label{T:outer_aut}
    Let $\phi\in \operatorname{Aut}(\Q)$ satisfy $\phi(u)=\lambda u^*$ for some $\lambda\in \mathbb{Z}(2^\infty)$. Then $\phi$ is an outer automorphism.
\end{theo}

\begin{proof}
    Recall that $K_1(\Q)\cong \mathbb{Z}$ and that the $K_1$-class of $u^n$ is given by $[u^n]=n$ for $n\in\mathbb{Z}$. Suppose, towards a contradiction, that $\phi$ is an inner automorphism. Then $K_1(\phi)=\operatorname{id}_{K_1(\Q)}$. Hence,
    \[
    [u]=K_1(\phi)([u])=[\phi(u)]=[\lambda u^*]=[u^*]=-[u],
    \]
    which is impossible since $[u]$ is a generator of $K_1(\Q)\cong\mathbb{Z}$. Therefore, $\phi$ is outer.
\end{proof}

The following corollary characterizes the inner automorphisms of $\Q$ that globally preserve $C^*(u)$.

\begin{cor}\label{C:autp}
Every $\phi \in \operatorname{Inn}(\Q, C^*(u))$ is of the form
\[
\operatorname{Ad}(U_\xi)\circ \lambda_{V\varphi(V^*)}
\]
for some $V\in\mathcal{U}(C^*(u))$ and $\xi\in\mathbb{Z}(2^\infty)$. Equivalently,
\(
\phi=\operatorname{Ad}(U_\xi V)
\).
\end{cor}

\begin{proof}
Let $\phi \in \operatorname{Inn}(\Q, C^*(u))$. By Theorem \ref{T:autp}, $\phi$ is either of the form
\[
\operatorname{Ad}(U_\xi)\circ \beta_V
\quad \text{or} \quad
\sigma \circ \operatorname{Ad}(U_\xi)\circ \beta_V
\]
for some $W\in\mathcal{U}(C^*(u))$ and $\xi\in\mathbb{Z}(2^\infty)$. By Theorem \ref{T:outer_aut}, the latter case is impossible, since
\[
(\sigma \circ \operatorname{Ad}(U_\xi)\circ \beta_W)(u)=\xi u^*,
\]
and hence $\sigma \circ \operatorname{Ad}(U_\xi)\circ \beta_W$ is outer. Therefore,
\(
\phi=\operatorname{Ad}(U_\xi)\circ \beta_V.
\)
Moreover, $\operatorname{Ad}(U_\xi)\circ \beta_V$ is inner if and only if $\beta_W$ is inner. This occurs precisely when
\(
W=V\varphi(V^*)
\)
for some $V\in\mathcal{U}(C^*(u))$. Hence,
\[
\phi=\operatorname{Ad}(U_\xi)\circ \lambda_{V\varphi(V^*)},
\]
as required.
\end{proof}

Before stating a consequence of the preceding result, recall the unitary normalizer normalizer of $C^*(u)$ in $\Q$:
\[
\mathcal{N}_{C^*(u)}(\Q)
\coloneqq
\{\,w\in\mathcal{U}(\Q):
wC^*(u)w^*=C^*(u)\,\}.
\]
We now determine $\mathcal{N}_{C^*(u)}(\Q)$, thereby strengthening \cite[Proposition 3.1]{ACR20} and resolving \cite[Question 9.27]{ACR21}.
\begin{cor}\label{C:normalizer}
Every $w\in\mathcal{N}_{C^*(u)}(\Q)$ can be written as
\(
w=U_\xi V
\)
for some $V\in\mathcal{U}(C^*(u))$ and $\xi\in\mathbb{Z}(2^\infty)$. Equivalently,
\[
\mathcal{N}_{C^*(u)}(\Q)
=
\left\langle
\mathcal{U}(C^*(u)),\,
U_\xi:\xi\in\mathbb{Z}(2^\infty)
\right\rangle.
\]
\end{cor}

\begin{proof}
Let $w\in\mathcal{N}_{C^*(u)}(\Q)$. Then $\operatorname{Ad}(w)\in\operatorname{Inn}(\Q,C^*(u))$. By Corollary~\ref{C:autp},
\[
\operatorname{Ad}(w)=\operatorname{Ad}(U_\xi V_1)
\]
for some $V_1\in\mathcal{U}(C^*(u))$ and $\xi\in\mathbb{Z}(2^\infty)$. Hence,
\[
wxw^*=U_\xi V_1x(U_\xi V_1)^* \quad \text{for all }x\in\Q.
\]
It follows that
\[
V_1^*U_\xi^*w\in\Q'\cap\Q.
\]
Since $\Q'\cap\Q=\mathbb{C}1$ by \cite[Corollary~3.12]{ACR18}, there exists $\lambda\in\mathbb{T}$ such that
\[
V_1^*U_\xi^*w=\lambda1.
\]
Therefore,
\[
w=U_\xi(\lambda V_1)=U_\xi V,
\]
where $V=\lambda V_1\in\mathcal{U}(C^*(u))$.
\end{proof}

\subsection{The Weyl group $\mathcal{W}(\Q, C^*(u))$ and the topological full group $[[\Gamma(\mathbb{T},\gamma)]]$}
%%%%%%%%%%
Recall that the Weyl group for the pair $(\Q, C^*(u))$ is
\[
\mathcal{W}(\Q, C^*(u))
=
\operatorname{Aut}(\Q, C^*(u))
\big/
\operatorname{Aut}_{C^*(u)}(\Q),
\]
and that its outer counterpart is
\[
\mathcal{W}_\pi(\Q, C^*(u))
=
\operatorname{Out}(\Q, C^*(u))
\big/
\operatorname{Out}_{C^*(u)}(\Q).
\]
Using the results established in this section so far, we are now ready to give a complete description of $\mathcal{W}(\Q,C^*(u))$ and, as a by-product, describe the topological full group $[[\Gamma(\mathbb{T},\gamma)]]$. We refer to \cite{Matui12} for background on topological full groups.
\begin{cor}\label{C:Weyl2}
 We have
\[
\mathcal{W}(\Q, C^*(u))
=
\left\langle
[\sigma],\,
[\operatorname{Ad}(U_\xi)]
:\xi\in\mathbb{Z}(2^\infty)
\right\rangle
\cong
\mathbb{Z}(2^\infty)\rtimes\mathbb{Z}/2\mathbb{Z},
\]
and
\[
\mathcal{W}_\pi(\Q, C^*(u))
=
\{[{\rm id}],[\sigma]\}
\cong
\mathbb{Z}/2\mathbb{Z}.
\]
\end{cor}
\begin{proof}
The proof of Theorem~\ref{T:autp} shows that
\[
\mathcal{W}(\Q, C^*(u))
=
\left\langle
[\sigma],\,
[\operatorname{Ad}(U_\xi)]
:\xi\in\mathbb{Z}(2^\infty)
\right\rangle.
\]

To identify the group structure, note that $\operatorname{Ad}(U_\xi)(u)=\xi u$ by \cite{ACR18}. Hence,
\[
\sigma\circ\operatorname{Ad}(U_\xi)\circ\sigma(u)
=
\operatorname{Ad}(U_{\bar{\xi}})(u).
\]
It follows from Corollary~\ref{C:u-uni} and Remark~\ref{R:u-uni} that
\[
[\sigma\circ\operatorname{Ad}(U_\xi)\circ\sigma]
=
[\operatorname{Ad}(U_{\bar{\xi}})].
\]
Thus, conjugation by $[\sigma]$ acts on $\mathbb{Z}(2^\infty)$ by inversion, and therefore
\[
\mathcal{W}(\Q,C^*(u))
\cong
\mathbb{Z}(2^\infty)\rtimes\mathbb{Z}/2\mathbb{Z}.
\]

The description of $\mathcal{W}_\pi(\Q,C^*(u))$ follows immediately from Theorem~\ref{T:outer_aut} and Corollary~\ref{C:autp}.
\end{proof}
We can now easily describe the topological full group of 
$\Gamma(\mathbb T, \gamma)$.
\begin{cor}
\label{C:TFG}
$[[\Gamma(\mathbb T, \gamma]]\cong \mathbb{Z}(2^\infty)$.
\end{cor}

\begin{proof}
By \cite[Proposition 5.7]{Matui12},
\[
[[\Gamma(\mathbb T, \gamma)]]\cong \operatorname{Inn}(\Q,C^*(u))/
\operatorname{Inn}_{C^*(u)}(\Q).
\]
It follows from Corollary~\ref{C:autp} that every $\phi\in\operatorname{Inn}(\Q,C^*(u))$ is of the form
\[
\phi=\operatorname{Ad}(U_\xi V)
\]
for some $V\in\mathcal{U}(C^*(u))$ and $\xi\in\mathbb{Z}(2^\infty)$. Moreover, $\phi$ belongs to $\operatorname{Inn}_{C^*(u)}(\Q)$ precisely when it is of the form $\operatorname{Ad}(V)$. Hence, the quotient identifies the class of $\operatorname{Ad}(U_\xi V)$ with the class of $\operatorname{Ad}(U_\xi)$, and therefore
\[
[[\Gamma(\mathbb T, \gamma)]]\cong \mathbb{Z}(2^\infty),
\]
as desired.
\end{proof}

We conclude the paper with the following remark.
\begin{rmk}
It is worth mentioning that some results in this paper may be extended to broader classes of $C^*$-algebras.
\end{rmk}

\bibliographystyle{elsarticle-harv} 
 \bibliography{mybib}

\begin{thebibliography}{40}
\expandafter\ifx\csname natexlab\endcsname\relax\def\natexlab#1{#1}\fi
\providecommand{\url}[1]{\texttt{#1}}
\providecommand{\href}[2]{#2}
\providecommand{\path}[1]{#1}
\providecommand{\DOIprefix}{doi:}
\providecommand{\ArXivprefix}{arXiv:}
\providecommand{\URLprefix}{URL: }
\providecommand{\Pubmedprefix}{pmid:}
\providecommand{\doi}[1]{\href{http://dx.doi.org/#1}{\path{#1}}}
\providecommand{\Pubmed}[1]{\href{pmid:#1}{\path{#1}}}
\providecommand{\bibinfo}[2]{#2}
\ifx\xfnm\relax \def\xfnm[#1]{\unskip,\space#1}\fi
%Type = Article
\bibitem[{Aiello et~al.(2018a)Aiello, Conti and Rossi}]{ACR18b}
\bibinfo{author}{Aiello, V.}, \bibinfo{author}{Conti, R.}, \bibinfo{author}{Rossi, S.}, \bibinfo{year}{2018}a.
\newblock \bibinfo{title}{Diagonal automorphisms of the 2-adic ring {$C^*$}-algebra}.
\newblock \bibinfo{journal}{Q. J. Math.} \bibinfo{volume}{69}, \bibinfo{pages}{815--833}.
\newblock \URLprefix \url{https://doi.org/10.1093/qmath/hax064}, \DOIprefix\doi{10.1093/qmath/hax064}.
%Type = Article
\bibitem[{Aiello et~al.(2018b)Aiello, Conti and Rossi}]{ACR18}
\bibinfo{author}{Aiello, V.}, \bibinfo{author}{Conti, R.}, \bibinfo{author}{Rossi, S.}, \bibinfo{year}{2018}b.
\newblock \bibinfo{title}{A look at the inner structure of the 2-adic ring {$C^*$}-algebra and its automorphism groups}.
\newblock \bibinfo{journal}{Publ. Res. Inst. Math. Sci.} \bibinfo{volume}{54}, \bibinfo{pages}{45--87}.
\newblock \URLprefix \url{https://doi.org/10.4171/PRIMS/54-1-2}, \DOIprefix\doi{10.4171/PRIMS/54-1-2}.
%Type = Article
\bibitem[{Aiello et~al.(2020)Aiello, Conti and Rossi}]{ACR20}
\bibinfo{author}{Aiello, V.}, \bibinfo{author}{Conti, R.}, \bibinfo{author}{Rossi, S.}, \bibinfo{year}{2020}.
\newblock \bibinfo{title}{Normalizers and permutative endomorphisms of the 2-adic ring {$C^*$}-algebra}.
\newblock \bibinfo{journal}{J. Math. Anal. Appl.} \bibinfo{volume}{481}, \bibinfo{pages}{123395, 25}.
\newblock \URLprefix \url{https://doi.org/10.1016/j.jmaa.2019.123395}, \DOIprefix\doi{10.1016/j.jmaa.2019.123395}.
%Type = Article
\bibitem[{Aiello et~al.(2021)Aiello, Conti and Rossi}]{ACR21}
\bibinfo{author}{Aiello, V.}, \bibinfo{author}{Conti, R.}, \bibinfo{author}{Rossi, S.}, \bibinfo{year}{2021}.
\newblock \bibinfo{title}{A hitchhiker's guide to endomorphisms and automorphisms of {C}untz algebras}.
\newblock \bibinfo{journal}{Rend. Mat. Appl. (7)} \bibinfo{volume}{42}, \bibinfo{pages}{61--162}.
%Type = Article
\bibitem[{Barlak et~al.(2018)Barlak, Omland and Stammeier}]{BOS18}
\bibinfo{author}{Barlak, S.}, \bibinfo{author}{Omland, T.}, \bibinfo{author}{Stammeier, N.}, \bibinfo{year}{2018}.
\newblock \bibinfo{title}{On the {$K$}-theory of {$C^{\ast}$}-algebras arising from integral dynamics}.
\newblock \bibinfo{journal}{Ergodic Theory Dynam. Systems} \bibinfo{volume}{38}, \bibinfo{pages}{832--862}.
\newblock \URLprefix \url{https://doi.org/10.1017/etds.2016.63}, \DOIprefix\doi{10.1017/etds.2016.63}.
%Type = Article
\bibitem[{Bassi and Conti(forthcoming)}]{BC2025}
\bibinfo{author}{Bassi, J.}, \bibinfo{author}{Conti, R.}, \bibinfo{year}{forthcoming}.
\newblock \bibinfo{title}{On the inclusion $\mathcal{O}_2 \subset \mathcal{Q}_2$}.
\newblock \bibinfo{journal}{Kyoto Journal of Mathematics\hskip -0.02cm} , \bibinfo{pages}{to appear}.
%Type = Article
\bibitem[{Bratteli and Jorgensen(1999)}]{BJ99}
\bibinfo{author}{Bratteli, O.}, \bibinfo{author}{Jorgensen, P.E.T.}, \bibinfo{year}{1999}.
\newblock \bibinfo{title}{Iterated function systems and permutation representations of the {C}untz algebra}.
\newblock \bibinfo{journal}{Mem. Amer. Math. Soc.} \bibinfo{volume}{139}, \bibinfo{pages}{x+89}.
\newblock \URLprefix \url{https://doi.org/10.1090/memo/0663}, \DOIprefix\doi{10.1090/memo/0663}.
%Type = Article
\bibitem[{Brownlowe et~al.(2014)Brownlowe, Ramagge, Robertson and Whittaker}]{BRRW14}
\bibinfo{author}{Brownlowe, N.}, \bibinfo{author}{Ramagge, J.}, \bibinfo{author}{Robertson, D.}, \bibinfo{author}{Whittaker, M.F.}, \bibinfo{year}{2014}.
\newblock \bibinfo{title}{Zappa-{S}z\'ep products of semigroups and their {$C^\ast$}-algebras}.
\newblock \bibinfo{journal}{J. Funct. Anal.} \bibinfo{volume}{266}, \bibinfo{pages}{3937--3967}.
\newblock \URLprefix \url{https://doi.org/10.1016/j.jfa.2013.12.025}, \DOIprefix\doi{10.1016/j.jfa.2013.12.025}.
%Type = Article
\bibitem[{Choi and Latr\'emoli\`ere(2012)}]{CL12}
\bibinfo{author}{Choi, M.D.}, \bibinfo{author}{Latr\'emoli\`ere, F.}, \bibinfo{year}{2012}.
\newblock \bibinfo{title}{Symmetry in the {C}untz algebra on two generators}.
\newblock \bibinfo{journal}{J. Math. Anal. Appl.} \bibinfo{volume}{387}, \bibinfo{pages}{1050--1060}.
\newblock \URLprefix \url{https://doi.org/10.1016/j.jmaa.2011.10.008}, \DOIprefix\doi{10.1016/j.jmaa.2011.10.008}.
%Type = Article
\bibitem[{Clark et~al.(2016)Clark, an~Huef and Raeburn}]{CaHR16}
\bibinfo{author}{Clark, L.O.}, \bibinfo{author}{an~Huef, A.}, \bibinfo{author}{Raeburn, I.}, \bibinfo{year}{2016}.
\newblock \bibinfo{title}{Phase transitions on the {T}oeplitz algebras of {B}aumslag-{S}olitar semigroups}.
\newblock \bibinfo{journal}{Indiana Univ. Math. J.} \bibinfo{volume}{65}, \bibinfo{pages}{2137--2173}.
\newblock \URLprefix \url{https://doi.org/10.1512/iumj.2016.65.5934}, \DOIprefix\doi{10.1512/iumj.2016.65.5934}.
%Type = Article
\bibitem[{Conti et~al.(2012a)Conti, Hong and Szyma\'nski}]{CHS12a}
\bibinfo{author}{Conti, R.}, \bibinfo{author}{Hong, J.H.}, \bibinfo{author}{Szyma\'nski, W.}, \bibinfo{year}{2012}a.
\newblock \bibinfo{title}{The restricted {W}eyl group of the {C}untz algebra and shift endomorphisms}.
\newblock \bibinfo{journal}{J. Reine Angew. Math.} \bibinfo{volume}{667}, \bibinfo{pages}{177--191}.
\newblock \URLprefix \url{https://doi.org/10.1515/crelle.2011.125}, \DOIprefix\doi{10.1515/crelle.2011.125}.
%Type = Article
\bibitem[{Conti et~al.(2012b)Conti, Hong and Szyma\'nski}]{CHS12}
\bibinfo{author}{Conti, R.}, \bibinfo{author}{Hong, J.H.}, \bibinfo{author}{Szyma\'nski, W.}, \bibinfo{year}{2012}b.
\newblock \bibinfo{title}{The {W}eyl group of the {C}untz algebra}.
\newblock \bibinfo{journal}{Adv. Math.} \bibinfo{volume}{231}, \bibinfo{pages}{3147--3161}.
\newblock \URLprefix \url{https://doi.org/10.1016/j.aim.2012.09.003}, \DOIprefix\doi{10.1016/j.aim.2012.09.003}.
%Type = Article
\bibitem[{Conti et~al.(2010)Conti, R{\o}rdam and Szyma\'nski}]{CRS10}
\bibinfo{author}{Conti, R.}, \bibinfo{author}{R{\o}rdam, M.}, \bibinfo{author}{Szyma\'nski, W.}, \bibinfo{year}{2010}.
\newblock \bibinfo{title}{Endomorphisms of {$\mathcal{O}_n$} which preserve the canonical {UHF}-subalgebra}.
\newblock \bibinfo{journal}{J. Funct. Anal.} \bibinfo{volume}{259}, \bibinfo{pages}{602--617}.
\newblock \URLprefix \url{https://doi.org/10.1016/j.jfa.2010.03.027}, \DOIprefix\doi{10.1016/j.jfa.2010.03.027}.
%Type = Article
\bibitem[{Conti and Szyma\'nski(2011)}]{CS11}
\bibinfo{author}{Conti, R.}, \bibinfo{author}{Szyma\'nski, W.}, \bibinfo{year}{2011}.
\newblock \bibinfo{title}{Labeled trees and localized automorphisms of the {C}untz algebras}.
\newblock \bibinfo{journal}{Trans. Amer. Math. Soc.} \bibinfo{volume}{363}, \bibinfo{pages}{5847--5870}.
\newblock \URLprefix \url{https://doi.org/10.1090/S0002-9947-2011-05234-7}, \DOIprefix\doi{10.1090/S0002-9947-2011-05234-7}.
%Type = Article
\bibitem[{Cuntz(1977)}]{Cuntz77}
\bibinfo{author}{Cuntz, J.}, \bibinfo{year}{1977}.
\newblock \bibinfo{title}{Simple {$C^*$}-algebras generated by isometries}.
\newblock \bibinfo{journal}{Comm. Math. Phys.} \bibinfo{volume}{57}, \bibinfo{pages}{173--185}.
\newblock \URLprefix \url{http://projecteuclid.org/euclid.cmp/1103901288}.
%Type = Article
\bibitem[{Cuntz(1980)}]{Cun80}
\bibinfo{author}{Cuntz, J.}, \bibinfo{year}{1980}.
\newblock \bibinfo{title}{Automorphisms of certain simple {$C^*$}-algebras}.
\newblock \bibinfo{journal}{Quantum fields-algebras, processes (Proc. Sympos., Univ. Bielefeld, Bielefeld, 1978), Springer, Vienna} , \bibinfo{pages}{187--196}.
%Type = Article
\bibitem[{Deaconu(1995)}]{Deaconu95}
\bibinfo{author}{Deaconu, V.}, \bibinfo{year}{1995}.
\newblock \bibinfo{title}{Groupoids associated with endomorphisms}.
\newblock \bibinfo{journal}{Trans. Amer. Math. Soc.} \bibinfo{volume}{347}, \bibinfo{pages}{1779--1786}.
\newblock \URLprefix \url{https://doi.org/10.2307/2154972}, \DOIprefix\doi{10.2307/2154972}.
%Type = Article
\bibitem[{Doplicher and Roberts(1990)}]{DR90}
\bibinfo{author}{Doplicher, S.}, \bibinfo{author}{Roberts, J.E.}, \bibinfo{year}{1990}.
\newblock \bibinfo{title}{Why there is a field algebra with a compact gauge group describing the superselection structure in particle physics}.
\newblock \bibinfo{journal}{Comm. Math. Phys.} \bibinfo{volume}{131}, \bibinfo{pages}{51--107}.
\newblock \URLprefix \url{http://projecteuclid.org/euclid.cmp/1104200703}.
%Type = Article
\bibitem[{Exel and Vershik(2006)}]{EV06}
\bibinfo{author}{Exel, R.}, \bibinfo{author}{Vershik, A.}, \bibinfo{year}{2006}.
\newblock \bibinfo{title}{{$C^\ast$}-algebras of irreversible dynamical systems}.
\newblock \bibinfo{journal}{Canad. J. Math.} \bibinfo{volume}{58}, \bibinfo{pages}{39--63}.
\newblock \URLprefix \url{https://doi.org/10.4153/CJM-2006-003-x}, \DOIprefix\doi{10.4153/CJM-2006-003-x}.
%Type = Book
\bibitem[{Folland(2016)}]{Fol95}
\bibinfo{author}{Folland, G.B.}, \bibinfo{year}{2016}.
\newblock \bibinfo{title}{A course in abstract harmonic analysis}.
\newblock Textbooks in Mathematics. \bibinfo{edition}{second} ed., \bibinfo{publisher}{CRC Press, Boca Raton, FL}.
%Type = Book
\bibitem[{Johnson(1972)}]{Johnson72}
\bibinfo{author}{Johnson, B.E.}, \bibinfo{year}{1972}.
\newblock \bibinfo{title}{Cohomology in {B}anach algebras}. volume \bibinfo{volume}{No. 127} of \textit{\bibinfo{series}{Memoirs of the American Mathematical Society}}.
\newblock \bibinfo{publisher}{American Mathematical Society, Providence, RI}.
%Type = Article
\bibitem[{Katsura(2008)}]{Katsura08}
\bibinfo{author}{Katsura, T.}, \bibinfo{year}{2008}.
\newblock \bibinfo{title}{A class of {$C^*$}-algebras generalizing both graph algebras and homeomorphism {$C^*$}-algebras. {IV}. {P}ure infiniteness}.
\newblock \bibinfo{journal}{J. Funct. Anal.} \bibinfo{volume}{254}, \bibinfo{pages}{1161--1187}.
\newblock \URLprefix \url{https://doi.org/10.1016/j.jfa.2007.11.014}, \DOIprefix\doi{10.1016/j.jfa.2007.11.014}.
%Type = Article
\bibitem[{Komura(2025)}]{Komura25}
\bibinfo{author}{Komura, F.}, \bibinfo{year}{2025}.
\newblock \bibinfo{title}{Weyl groups of groupoid {$\rm C^*$}-algebras}.
\newblock \bibinfo{journal}{J. Funct. Anal.} \bibinfo{volume}{289}, \bibinfo{pages}{Paper No. 111166, 60}.
\newblock \URLprefix \url{https://doi.org/10.1016/j.jfa.2025.111166}, \DOIprefix\doi{10.1016/j.jfa.2025.111166}.
%Type = Article
\bibitem[{Laca et~al.(2011)Laca, Raeburn and Ramagge}]{LRR11}
\bibinfo{author}{Laca, M.}, \bibinfo{author}{Raeburn, I.}, \bibinfo{author}{Ramagge, J.}, \bibinfo{year}{2011}.
\newblock \bibinfo{title}{Phase transition on {E}xel crossed products associated to dilation matrices}.
\newblock \bibinfo{journal}{J. Funct. Anal.} \bibinfo{volume}{261}, \bibinfo{pages}{3633--3664}.
\newblock \URLprefix \url{https://doi.org/10.1016/j.jfa.2011.08.015}, \DOIprefix\doi{10.1016/j.jfa.2011.08.015}.
%Type = Article
\bibitem[{Larsen and Li(2012)}]{LL12}
\bibinfo{author}{Larsen, N.S.}, \bibinfo{author}{Li, X.}, \bibinfo{year}{2012}.
\newblock \bibinfo{title}{The 2-adic ring {$C^\ast$}-algebra of the integers and its representations}.
\newblock \bibinfo{journal}{J. Funct. Anal.} \bibinfo{volume}{262}, \bibinfo{pages}{1392--1426}.
\newblock \URLprefix \url{https://doi.org/10.1016/j.jfa.2011.11.008}, \DOIprefix\doi{10.1016/j.jfa.2011.11.008}.
%Type = Article
\bibitem[{Li and Yang(2019)}]{LY19}
\bibinfo{author}{Li, H.}, \bibinfo{author}{Yang, D.}, \bibinfo{year}{2019}.
\newblock \bibinfo{title}{Boundary quotient {$\rm C^*$}-algebras of products of odometers}.
\newblock \bibinfo{journal}{Canad. J. Math.} \bibinfo{volume}{71}, \bibinfo{pages}{183--212}.
\newblock \URLprefix \url{https://doi.org/10.4153/cjm-2017-034-5}, \DOIprefix\doi{10.4153/cjm-2017-034-5}.
%Type = Article
\bibitem[{Li and Yang(2021)}]{LY21}
\bibinfo{author}{Li, H.}, \bibinfo{author}{Yang, D.}, \bibinfo{year}{2021}.
\newblock \bibinfo{title}{Self-similar {$k$}-graph {$C^*$}-algebras}.
\newblock \bibinfo{journal}{Int. Math. Res. Not. IMRN} , \bibinfo{pages}{11270--11305}\URLprefix \url{https://doi.org/10.1093/imrn/rnz146}, \DOIprefix\doi{10.1093/imrn/rnz146}.
%Type = Article
\bibitem[{Matsumoto and Tomiyama(1993)}]{MT93}
\bibinfo{author}{Matsumoto, K.}, \bibinfo{author}{Tomiyama, J.}, \bibinfo{year}{1993}.
\newblock \bibinfo{title}{Outer automorphisms on {C}untz algebras}.
\newblock \bibinfo{journal}{Bull. London Math. Soc.} \bibinfo{volume}{25}, \bibinfo{pages}{64--66}.
\newblock \URLprefix \url{https://doi.org/10.1112/blms/25.1.64}, \DOIprefix\doi{10.1112/blms/25.1.64}.
%Type = Article
\bibitem[{Matui(2012)}]{Matui12}
\bibinfo{author}{Matui, H.}, \bibinfo{year}{2012}.
\newblock \bibinfo{title}{Homology and topological full groups of \'etale groupoids on totally disconnected spaces}.
\newblock \bibinfo{journal}{Proc. Lond. Math. Soc. (3)} \bibinfo{volume}{104}, \bibinfo{pages}{27--56}.
\newblock \URLprefix \url{https://doi.org/10.1112/plms/pdr029}, \DOIprefix\doi{10.1112/plms/pdr029}.
%Type = Article
\bibitem[{Oyetunbi and Yang(2026)}]{OY26}
\bibinfo{author}{Oyetunbi, D.}, \bibinfo{author}{Yang, D.}, \bibinfo{year}{2026}.
\newblock \bibinfo{title}{Maximality and symmetry related to the 2-adic ring {$C^*$}-algebra}.
\newblock \bibinfo{journal}{J. Funct. Anal.} \bibinfo{volume}{290}, \bibinfo{pages}{Paper No. 111349, 34}.
\newblock \URLprefix \url{https://doi.org/10.1016/j.jfa.2026.111349}, \DOIprefix\doi{10.1016/j.jfa.2026.111349}.
%Type = Article
\bibitem[{Renault(2008)}]{Ren08}
\bibinfo{author}{Renault, J.}, \bibinfo{year}{2008}.
\newblock \bibinfo{title}{Cartan subalgebras in {$C^*$}-algebras}.
\newblock \bibinfo{journal}{Irish Math. Soc. Bull.} , \bibinfo{pages}{29--63}.
%Type = Book
\bibitem[{Robert(2000)}]{Rob00}
\bibinfo{author}{Robert, A.M.}, \bibinfo{year}{2000}.
\newblock \bibinfo{title}{A course in {$p$}-adic analysis}. volume \bibinfo{volume}{198} of \textit{\bibinfo{series}{Graduate Texts in Mathematics}}.
\newblock \bibinfo{publisher}{Springer-Verlag, New York}.
\newblock \URLprefix \url{https://doi.org/10.1007/978-1-4757-3254-2}, \DOIprefix\doi{10.1007/978-1-4757-3254-2}.
%Type = Article
\bibitem[{Rosenberg(1977)}]{Rosenberg77}
\bibinfo{author}{Rosenberg, J.}, \bibinfo{year}{1977}.
\newblock \bibinfo{title}{Amenability of crossed products of {$C\sp*$}-algebras}.
\newblock \bibinfo{journal}{Comm. Math. Phys.} \bibinfo{volume}{57}, \bibinfo{pages}{187--191}.
\newblock \URLprefix \url{http://projecteuclid.org/euclid.cmp/1103901289}.
%Type = Book
\bibitem[{Rudin(1976)}]{Rudin76}
\bibinfo{author}{Rudin, W.}, \bibinfo{year}{1976}.
\newblock \bibinfo{title}{Principles of mathematical analysis}.
\newblock International Series in Pure and Applied Mathematics. \bibinfo{edition}{third} ed., \bibinfo{publisher}{McGraw-Hill Book Co., New York-Auckland-D\"usseldorf}.
%Type = Phdthesis
\bibitem[{Somasunderam(2020)}]{Som20}
\bibinfo{author}{Somasunderam, N.}, \bibinfo{year}{2020}.
\newblock \bibinfo{title}{Fourier Analysis and Equidistribution on the {$p$}-adic Integers}.
\newblock Ph.D. thesis. Oregon State University.
%Type = Article
\bibitem[{Spielberg(2012)}]{Spi12}
\bibinfo{author}{Spielberg, J.}, \bibinfo{year}{2012}.
\newblock \bibinfo{title}{{$C^\ast$}-algebras for categories of paths associated to the {B}aumslag-{S}olitar groups}.
\newblock \bibinfo{journal}{J. Lond. Math. Soc. (2)} \bibinfo{volume}{86}, \bibinfo{pages}{728--754}.
\newblock \URLprefix \url{https://doi.org/10.1112/jlms/jds025}, \DOIprefix\doi{10.1112/jlms/jds025}.
%Type = Article
\bibitem[{Taibleson(1967)}]{Tailbleson67}
\bibinfo{author}{Taibleson, M.H.}, \bibinfo{year}{1967}.
\newblock \bibinfo{title}{Fourier series on the ring of integers in a {$p$}-series field}.
\newblock \bibinfo{journal}{Bull. Amer. Math. Soc.} \bibinfo{volume}{73}, \bibinfo{pages}{623--629}.
\newblock \URLprefix \url{https://doi.org/10.1090/S0002-9904-1967-11801-9}, \DOIprefix\doi{10.1090/S0002-9904-1967-11801-9}.
%Type = Book
\bibitem[{Urba\'nski et~al.(2022)Urba\'nski, Roy and Munday}]{URS22}
\bibinfo{author}{Urba\'nski, M.}, \bibinfo{author}{Roy, M.}, \bibinfo{author}{Munday, S.}, \bibinfo{year}{2022}.
\newblock \bibinfo{title}{Non-Invertible Dynamical Systems}. volume \bibinfo{volume}{69/1} of \textit{\bibinfo{series}{Exposition in Mathematics}}.
\newblock \bibinfo{publisher}{Walter de Gruyter GmbH, Berlin/Boston}.
\newblock \bibinfo{note}{Volume 1: Ergodic Theory – Finite and Infinite, Thermodynamic Formalism, Symbolic Dynamics and Distance Expanding Maps}.
%Type = Article
\bibitem[{Valente and Yang(2025)}]{VY25}
\bibinfo{author}{Valente, R.}, \bibinfo{author}{Yang, D.}, \bibinfo{year}{2025}.
\newblock \bibinfo{title}{Semigroups of self-similar actions and higher rank baumslag-solitar semigroups}.
\newblock \bibinfo{journal}{Proc. R. Soc. Edinb. A} \DOIprefix\doi{10.1017/prm.2025.10053}.
%Type = Article
\bibitem[{Yang(2010)}]{Yan10}
\bibinfo{author}{Yang, D.}, \bibinfo{year}{2010}.
\newblock \bibinfo{title}{Endomorphisms and modular theory of 2-graph {$C^*$}-algebras}.
\newblock \bibinfo{journal}{Indiana Univ. Math. J.} \bibinfo{volume}{59}, \bibinfo{pages}{495--520}.
\newblock \URLprefix \url{https://doi.org/10.1512/iumj.2010.59.3973}, \DOIprefix\doi{10.1512/iumj.2010.59.3973}.

\end{thebibliography}



%\begin{thebibliography}{00}

%% For numbered reference style

%\bibitem{lamport94}
%  \textit{\LaTeX: a document preparation system},
%  Addison Wesley, Massachusetts,
%  1994
%\end{thebibliography}
\end{document}